\documentclass[12pt]{amsart}

\usepackage{amscd,amsmath}
\usepackage{mathrsfs}
\usepackage{amsfonts}
\usepackage{amssymb}
\usepackage{enumerate}
\usepackage{setspace}
\usepackage{color}
\usepackage{nicefrac}
\usepackage[normalem]{ulem}
\usepackage{mathabx}

\usepackage{amsthm}
\usepackage{amssymb,latexsym,graphics,enumerate}
\usepackage[mathscr]{eucal}
\usepackage{amsmath,amsfonts,amsthm,amssymb}
\usepackage[left=1in,right=1in,top=1in,bottom=1in]{geometry}
\usepackage{tikz}

\usepackage[colorlinks=true,
            linkcolor=red,
            urlcolor=blue,
            citecolor=gray]{hyperref}

\usepackage[page,header]{appendix}
\usepackage{titletoc}

\numberwithin{equation}{section}
\newtheorem{theorem}{Theorem}[section]
\newtheorem{lemma}{Lemma}[section]
\newtheorem{corollary}{Corollary}[section]
\newtheorem{proposition}{Proposition}[section]

\newtheorem{remark}{Remark}[section]
\newtheorem{assumption}{Assumption}[section]

\def\bu{{\bf u}}
\def\bx{{\bf x}}
\def\bX{{\bf X}}
\def\by{{\bf y}}
\def\bz{{\bf z}}

\def\bv{{\bf v}}

\def\cE{\mathcal{E}}
\def\delE{\delta E}
\def\delcE{\delta\cE}
\def\cX{X}
\def\cXtau{\mathcal{X}} 
\def\kr{\delta_r} 
\def\kapta{\kappa} 
\def\lamta{\eta} 
\def\wU{U}
\def\Phij{\Phi_{ij}}
\def\aL{{\mathbf A}}
\newcommand\braket[1]{{\langle{#1}\rangle}}
\def\brr{\braket{r}}

\def\brt{\braket{t}}
\def\brx{\braket{\bx}}
\def\brz{\braket{\bz}}

\def\cof{c}

\def\ds{\displaystyle} 
\def\om{\omega}
\def\omp{\omega'}

\newcommand{\rd}{\,\mathrm{d}}
\newcommand{\rD}{\,\mathrm{D}}

\renewcommand{\geq}{\geqslant}
\renewcommand{\ge}{\geqslant}
\renewcommand{\leq}{\leqslant}
\renewcommand{\le}{\leqslant}

\begin{document}

\title[Anticipation breeds alignment]
{Anticipation breeds alignment}

\author{Ruiwen Shu}
\address{\noindent\newline Ruiwen Shu\newline
Department of Mathematics\newline
and\newline
Center for Scientific Computation and Mathematical Modeling (CSCAMM)\newline
University of Maryland, College Park MD 20742}
\email{rshu@cscamm.umd.edu}

\author{Eitan Tadmor}
\address{\noindent\newline Eitan Tadmor\newline
Department of Mathematics, Institute for Physical Sciences \& Technology (IPST)\newline
\and\newline
Center for Scientific Computation and Mathematical Modeling (CSCAMM)\newline
University of Maryland, College Park MD 20742}
\email{tadmor@cscamm.umd.edu}

\date{\today}



\thanks{\textbf{Acknowledgment.} Research was supported  by NSF grants DMS16-13911, RNMS11-07444 (KI-Net) and ONR grant N00014-1812465. ET thanks the hospitality of the Institut of Mittag-Leffler during  fall 2018 visit which initiated this work, and of the Laboratoire Jacques-Louis Lions   in Sorbonne University during spring 2019, with  support through ERC grant 740623 under the EU Horizon 2020, while concluding this work.}
\date{\today}

\begin{abstract}
 We study the large-time behavior of  systems driven by radial potentials,  which react to \emph{anticipated} positions, $\bx^\tau(t)=\bx(t)+\tau \bv(t)$, with  anticipation increment $\tau>0$. As a special case, such systems yield the celebrated Cucker-Smale model for alignment, coupled with pairwise interactions.
Viewed from this perspective, such anticipated-driven systems are expected to emerge into \emph{flocking} due to alignment of velocities, and   \emph{spatial concentration} due to confining potentials. We treat both the discrete dynamics and large crowd hydrodynamics, proving  the decisive role of anticipation in driving such systems with attractive potentials into velocity alignment and spatial concentration. We also study the concentration effect  near equilibrium for anticipated-based dynamics of pair of agents governed by attractive-repulsive  potentials.
\end{abstract}

\maketitle
\setcounter{tocdepth}{1}
\tableofcontents

\section{Introduction and statement of main results}
Consider the dynamical system
\[
\left\{ \ \ \begin{split}
& \dot{\bx}_i(t)=\bv_i(t) \\
& \dot{\bv}_i(t) = -\nabla_{i} {\mathcal H}_N(\bx_1^\tau, \ldots, \bx_N^\tau), \quad \bx_i^\tau= \bx_i^\tau(t) := \bx_i(t) + \tau \bv_i(t),
\end{split}\right.\qquad i=1,\dots,N.
\]
 When $\tau=0$, this is the classical $N$-particle dynamics for positions and velocities, $(\bx_i(t),\bv_i(t))\in (\mathbb{R}^d,\mathbb{R}^d)$, governed by the general Hamiltonian ${\mathcal H}_N(\cdots)$.
If we fix  a small  time step $\tau>0$, then  the system is not driven instantaneously but reacts to the positions $\bx^\tau(t)=\bx(t)+\tau \bv(t)$, \emph{anticipated} at time $t+\tau$, where $\tau$ is the anticipation time increment. Anticipation is a main feature in social dynamics of $N$-agent and $N$-player systems, \cite{GLL2010, MCEB2015,GTLW2017}.
A key feature in the the large time behavior of such anticipated dynamics  is the dissipation of the \emph{anticipated energy}
\[
\cE(t) = \frac{1}{2N}\sum_i |\bv_i|^2 + \frac{1}{N}{\mathcal H}_N(\bx_1^\tau,\ldots, \bx_N^\tau),
\]
at a rate given by 
\[
\begin{split}
\frac{\rd}{\rd{t}}\cE(t) =  \frac{1}{N}\sum_i \bv_i\cdot \dot{\bv}_i + \frac{1}{N}\sum_{i} \nabla_{i} {\mathcal H}_N(\bx_1^\tau,\ldots, \bx_N^\tau)\cdot (\bv_i + \tau \dot{\bv}_i) = -  \frac{\tau}{N}\sum_{i} | \dot{\bv}_i|^2, \quad \tau>0.
\end{split}
\]
We refer to the quantity on the right as the \emph{enstrophy} of the system.

\subsection{Pairwise interactions} In this work we study the anticiptaion dynamics of pairwise interactions

\begin{equation}\tag{AT}\label{eq:AT}
\left\{\ \ \begin{split}
& \dot{\bx}_i(t)=\bv_i(t) \\
& \dot{\bv}_i(t) = -\frac{1}{N}\sum_{j=1}^N \nabla U(|\bx_i^\tau-\bx_j^\tau|), \quad \bx_i^\tau=  \bx_i(t) + \tau \bv_i(t),
\end{split}\right.\qquad i=1,\dots,N,
\end{equation}
governed by a radial interaction potential $\wU(r), \, r=|\bx|$. This corresponds to the Hamiltonian
$\ds 
{\mathcal H}_N(\bx_1^\tau,\ldots,\bx_N^\tau)=\frac{1}{2N}\sum_{j,k}U(|\bx_j^\tau-\bx_k^\tau|)$, where the conservative $N$-body problem ($\tau=0$) is now 
replaced by $N$-agent dynamics with anticipated energy dissipation
\begin{equation}\label{energy}
\begin{split}
\frac{\rd}{\rd{t}}\cE(t) =   -  \frac{\tau}{N}\sum_{i} | \dot{\bv}_i|^2, \qquad \cE(t):=\frac{1}{2N}\sum_i |\bv_i|^2 +\frac{1}{2N^2}\sum_{i,j}U(|\bx_i^\tau-\bx_j^\tau|), \quad \tau>0.
\end{split}
\end{equation}
To gain a better insight  into \eqref{eq:AT}, expand in $\tau$ to obtain
\begin{equation}\tag{$\Phi$U}\label{eq:UF}
\left\{\ \ \begin{split}
& \dot{\bx}_i=\bv_i \\
& \dot{\bv}_i = \frac{\tau}{N}\sum_{j=1}^N \Phij(\bv_j-\bv_i) -\frac{1}{N}\sum_{j=1}^N \nabla U(|\bx_i-\bx_j|),
\end{split}\right.\qquad i=1,\dots,N.
\end{equation}
The anticipation  \eqref{eq:AT} is recovered in terms of the  Hessian, $\Phij=D^2U(|\bx_i^{\tau_{ij}}-\bx_j^{\tau_{ij}}|)$, evaluated at the mean-value anticipated times $\tau_{ij}(t)\in [0,\tau]$. Since these mean-valued times are not readily available, we will consider  \eqref{eq:UF} for a general class of  \emph{symmetric communication matrices} 
\[
{\sf \Phi}:= \left\{ \Phi(\cdot,\cdot)\in \text{Sym}_{d\times d} \ | \ \Phij:=\Phi\big((\bx_i,\bv_i),(\bx_j,\bv_j)\big)=\Phi_{ji}\right\}.
\]
The so-called \ref{eq:UF} system provides a unified framework for anticipation dynamics by coupling  general symmetric communication matrices, $\{\Phi \in {\sf \Phi}\}$, together with pairwise interactions induced by the potential $U$. 
In particular, the  anticipation dynamics  \eqref{eq:AT} yields upon linearization, the celebrated Cucker-Smale (CS) model \cite{CS2007a,CS2007b} --- a prototypical model for alignment dynamics in which 
$\max_{i,j}|\bv_i(t)-\bv_j(t)|\stackrel{t\rightarrow \infty}{\longrightarrow}0$.
 There is, however,  one distinct difference:  while the CS model is governed by a \emph{scalar} kernels involving geometric distances, $\Phij=\phi(|\bx_i-\bx_j|){\mathbb I}_{d\times d}$, \eqref{eq:UF} allows for  larger class of  communication protocols based on \emph{matrix} kernels, e.g., $\Phij=\Phi(\bx_i,\bx_j)$, with a possible dependence on   topological distances, \cite{ST2018b}. The flocking behavior of such \emph{matrix-based} CS models is proved in section \ref{subsec:CS}.
 
Viewed from the perspective of Cucker-Smale alignment dynamics, the large time behavior of  \eqref{eq:AT},\eqref{eq:UF}, is expected to emerge into \emph{flocking} due to alignment of velocities. Moreover, our study \cite{ST2019} shows that the  presence of  pairwise interactions leads (at least in the quadratic case $U(r) \sim \nicefrac{r^2}{2}$),  to  \emph{spatial concentration}. Similarly, spatial concentration   due to the confinement effect is expected for a large(r) class of pairwise interaction potentials $U$. The main purpose of this paper is to study the large time behavior of \eqref{eq:AT} and \eqref{eq:UF},  proving the decisive role of anticipation in driving the dynamics  of velocity alignment and spatial concentration.

We begin in section \ref{sec:convex}  with  the general system \eqref{eq:UF}. 
The basic bookkeeping associated with \eqref{eq:UF} quantifies its decay rate of  the (instantaneous) energy
\[
E(t) := \frac{1}{2N}\sum_i |\bv_i|^2 + \frac{1}{2N^2}\sum_{i,j} U(|\bx_i-\bx_j|),
\]
which is given by 
\begin{equation}\label{energy1}
\begin{split}
\frac{\rd}{\rd{t}}E(t) = & -  \frac{\tau}{2N^2}\sum_{i,j} (\bv_j-\bv_i)^\top \Phij(\bv_j-\bv_i).
\end{split}
\end{equation}
(This will  be contrasted  with the dissipation of anticipated energy  \eqref{energy} in section \ref{sec:attraction} below.)
To  explore the  enstrophy  on the right of \eqref{energy1} we need to further elaborate on properties of the  the communication matrices $\Phij=\Phi(\cdot,\cdot)$  vis a vis their relations with the potential $U$.
 
We study the dynamics  induced by potentials $U$ which are at least $C^2$. As a result, $\wU'(0)=0$, and we may assume $\wU(0)=0$ by adding a constant to it. 
We start by rewriting  the Hessian of the radial potential in the form
\begin{equation}\label{eq:comm}
D^2 U(|\bx_i-\bx_j|) = \frac{\wU'(r_{ij})}{r_{ij}} ({\mathbb I}-\widehat{\bx}_{ij}\widehat{\bx}_{ij}^\top) + \wU''(r_{ij})\widehat{\bx}_{ij}\widehat{\bx}_{ij}^\top,\quad r_{ij}:=|\bx_i-\bx_j|, \quad \widehat{\bx}_{ij}:=\frac{\bx_i-\bx_j}{r_{ij}},
\end{equation}
 and observe that the symmetric matrix $D^2U(|\bx_i-\bx_j)$ has  a single eigenvalue $\wU''(r_{ij})$ in the radial direction, $\bx_{i}-\bx_j$, and $d-1$ multiple  of the eigenvalues $\frac{\wU'(r_{ij})}{r_{ij}}$ in  tangential directions $(\bx_i-\bx_j)^\perp$. We now classify the classes of potentials we will be working with, according to the their short-range and long-range behavior. We begin by postulating  that the communication matrix $D^2U(|\bx_i-\bx_j|)$ is always bounded: 
 there exists a constant $A>0$ such that 
 \begin{equation}\label{eq:bounded}
|\wU''(r)| \leq A.
\end{equation}
\noindent
It follows that $\displaystyle |\wU'(r)| \leq \int_0^r |\wU''(s)|\rd{s} \leq \int_0^r A\rd{s} = Ar$
and hence $|D^2U(\cdot)|\leq A$. 
The assumed bound \eqref{eq:bounded} rules out the important class of kernels with short-range singularity (in both first- and second-order dynamics), e.g., \cite{Ja2014,Go2016, CCTT2016,PS2017,ST2017,CCP2017,Se2018,ST2018,DKRT2018,MMPZ2019},  which is left for a future study. Next, we distinguish between different  $U$'s according to their long range behavior. 

\subsection{Anticipation dynamics with convex potentials}
Recall that the flocking behavior of  CS model , see \eqref{eq:CS} below, is guaranteed for  \emph{scalar} communication kernels, $\Phi(r)=\phi(r){\mathbb I}$, which satisfy a so-called \emph{fat tail condition}, \cite{HL2009},\cite[Proposition 2.9]{MT2014}, 
\[
\int \phi(r)\rd{r} = \infty,
\]
or --- expressed in terms of its decay rate\footnote{Throughout this paper, we  use the notation $\brr^\beta := (1+r^2)^{\beta/2}$ for scalar $r$ and $\brz = \braket{|\bz|}$ for vectors $\bz$.}, $\phi(r)\sim \langle r\rangle^{-\beta}$ for $0\le \beta \le 1$. Since the anticipation model \eqref{eq:AT} can be viewed as a special case of \eqref{eq:UF} with  communication matrix evaluated at intermediate anticipated times, $\Phij= D^2U(|\bx_i^{\tau_{ij}}-\bx_j^{\tau_{ij}}|)$, it is natural to quantify the convexity of $U$ and positivity of $\Phi$ in terms of their `fat tail' decay. 

 \begin{assumption}[{\bf Convex potentials}]\label{ass:convex}
 There exist  constants $0< a<A$ and $\beta \in [0,1]$ such that 
 \begin{equation}\label{eq:convex}
a\brr^{-\beta} \leq \wU''(r) \leq A, \qquad 0\leq \beta \leq 1.
\end{equation}
 \end{assumption}
\noindent
It then follows that 
$\displaystyle 
\wU'(r) = \int_0^r \wU''(s)\rd{s} \ge \int_0^r a\langle s\rangle^{-\beta}\rd{s} \geq  a\langle r\rangle^{-\beta}r$, and hence $D^2U$ in \eqref{eq:comm} satisfies the fat tail condition $D^2U(|\bx|) \geq a\brx^{-\beta}$ with $0\leq \beta \leq 1$.

 \begin{assumption}[{\bf Positive kernels}]\label{ass:positive}
 There exist  constants $0<\phi_- <\phi_+$  such that 
 \begin{equation}\label{eq:positive}
\phi_-(\braket{\bx_i-\bx_j}+\braket{\bv_i-\bv_j})^{-\gamma} \leq \Phij \leq \phi_+, \qquad 0\leq \gamma <1.
\end{equation}
 \end{assumption}
 
Observe that \eqref{eq:UF} conserves momentum
\begin{equation}\label{momentum}
\left\{\begin{split}
& \dot{\widebar{\bx}}=\widebar{\bv},\qquad  \widebar{\bx}:=\frac{1}{N}\sum_i \bx_i,\\
& \dot{\widebar{\bv}}=0, \qquad  \widebar{\bv}:=\frac{1}{N}\sum_i \bv_i.
\end{split}\right.
\end{equation}
It follows that the mean velocity $\widebar{\bv}$  is constant in time,  $\widebar{\bv}(t)=\widebar{\bv}_0$, and hence $\widebar{\bx}(t) = \widebar{\bx}_0 + t\widebar{\bv}_0$.
Our first main result is expressed in terms of the \emph{energy fluctuations}
 \[
 \delE(t) := \frac{1}{2N}\sum_{i} |\bv_i-\widebar{\bv}|^2 + \frac{1}{2N^2}\sum_{i,j} U(|\bx_i-\bx_j|).
 \]
\begin{theorem}[{\bf Anticipation dynamics  \eqref{eq:UF} --- velocity alignment and spatial concentration}]\label{thm:thm1}
Consider the  anticipation dynamics \eqref{eq:UF}. Assume a bounded convex potential $U$ with fat-tail decay of order $\beta$,  \eqref{eq:convex}, and a symmetric kernel matrix $\Phi$ with a fat-tail decay of order $\gamma$,  \eqref{eq:positive}. If the decay parameters lie in the restricted range $3\beta +2\max\{\beta,\gamma\} <4$, then there is   sub-exponential decay of the  energy fluctuations
\begin{equation}\label{eq:Edecay}
\delE(t) \le Ce^{-t^{1-\lamta}},  \qquad  \lamta=\frac{2\max\{\beta,\gamma\}}{4-3\beta}<1.
\end{equation}
We conclude  that for large time, the  dynamics  concentrates in space with global velocity alignment at sub-exponential rate,
\begin{equation}
|\bx_i(t)-\widebar{\bx}(t)|\rightarrow 0,\quad |\bv_i(t)-\widebar{\bv}_0|\rightarrow 0, \qquad \widebar{\bx}(t)=\widebar{\bx}_0+t\widebar{\bv}_0.
\end{equation}
\end{theorem}
\noindent
The proof of Theorem \ref{thm:thm1} proceeds in two steps:\newline
 {\bf(i)} A uniform bound, outlined in lemma \ref{lem:convex} below, on maximal spread of positions $|\bx_i(t)|$,
 \begin{equation}\label{eq:bound1}
  \max_i |\bx_i(t)|  \le C_\infty\brt^{\frac{2}{4-3\beta}}, \quad \max_i|\bv_i(t)| \leq C_\infty\brt^{\frac{2-\beta}{4-3\beta}}, \qquad 0\leq \beta \leq 1.
 \end{equation}
{\bf (ii)} Observe that  in view of \eqref{momentum}, $\frac{\rd}{\rd{t}}\delE(t)=\frac{\rd}{\rd{t}}E(t)$. The energy dissipation \eqref{energy1} combined with the bounds \eqref{eq:positive},\eqref{eq:bound1} imply the decay of energy fluctuations
\[
\frac{\rd}{\rd{t}}\delE(t) =\frac{\rd}{\rd{t}}E(t) \lesssim -\frac{\tau}{2N}\brt^{-\frac{2\gamma}{4-3\beta}}\sum_{i} |\bv_i-\widebar{\bv}|^2.
\]
To close the last  bound we need a hypocoercivity argument carried out in section \ref{sec:convex}, which leads to the sub-exponential  decay \eqref{eq:Edecay}. The conclusion of sub-exponential flocking  $(\bx_i-\bx_j,\bv_i-\bv_i) \stackrel{t\rightarrow \infty}{\longrightarrow} 0$ follows, and naturally,  $(\bx_i,\bv_i) - (\widebar{\bx}(t),\widebar{\bv}_0)\rightarrow 0 $ since this is the only minimizer of  $\delE(t)$.

Since the anticipation dynamics \eqref{eq:AT} can be viewed as a special case of  \ref{eq:UF} system  with $\Phi=D^2U$ at intermediate anticipated time $\tau_{ij}$,  theorem \ref{thm:thm1} applies with $\gamma=\beta$.

\begin{corollary}[{\bf Anticipation dynamics \eqref{eq:AT} with convex potentials}]\label{cor:convex}
Consider the  anticipated dynamics \eqref{eq:AT} with   convex potential satisfying 
\[
a\langle r\rangle^{-\beta} \leq \wU''(r) \leq A, \qquad 0\leq \beta< \frac{4}{5}.
\]
Then there is sub-exponential decay of the  energy fluctuations
\begin{equation}
\delE(t) \le Ce^{-t^{1-{\lamta}}},\quad \lamta = \frac{2\beta}{4-3\beta}.
\end{equation}
The large time flocking behavior follows:  the  dynamics  concentrates in space with global velocity alignment at sub-exponential rate,
\begin{equation}
|\bx_i(t)-\widebar{\bx}(t)|\rightarrow 0,\quad |\bv_i(t)-\widebar{\bv}_0|\rightarrow 0
\end{equation}
\end{corollary}

\begin{remark}[{\bf Optimal result with improved fat-tail condition}]\label{rem:optimal}
Suppose we strengthen assumption \ref{ass:convex} with a more precise behavior of $U'' \sim \brr^{\beta}$, thus replacing \eqref{eq:convex} with the requirement that there exist constants $0<a < A$ and $\beta\in [0,1]$ such that
\begin{equation}\label{eq:convexop}
a\brr^{\beta} \leq U''(r), \quad U'(r) \leq A\brr^{1-\beta}, \qquad 0\leq \beta \leq 1.
\end{equation}
Then the  anticipation dynamics \eqref{eq:UF} with  a fat-tail kernel matrix $\Phi$  of order $\gamma$,  \eqref{eq:positive}, satisfies the   sub-exponential decay 
\begin{equation}\label{eq:Edecayop}
\delE(t) \le Ce^{-t^{1-\lamta}},  \qquad  \lamta=\min\Big\{1,\frac{2}{4-3\beta}\Big\}\cdot\max\{\beta,\gamma\}<1.
\end{equation}
This improved decay follows from the corresponding improvement of the uniform bound in lemma \ref{lem:convex} below which reads $\max_i|\bx_i(t)| \lesssim \brt$.
In the particular case of $\beta=\gamma$, we recover an improved corollary \ref{cor:convex} for anticipated dynamics \eqref{eq:AT}, where the  anticipated energy satisfies an optimal decay of order 
\[
\delta\cE(t) \leq  Ce^{-t^{1-\lamta}}, \qquad \lamta=\min\Big\{\frac{2\beta}{4-3\beta},\beta\Big\}<1.
\]
 \end{remark}

\subsection{Anticipation dynamics with purely attractive potential}
We now turn our attention to the main anticipation model \eqref{eq:AT}. 
 We already know the flocking behavior of \eqref{eq:AT} for convex potentials, from the general considerations of the \ref{eq:UF} system, summarized in corollary \ref{cor:convex}. In fact, the corresponding communication matrix of \eqref{eq:AT} prescribed in \eqref{eq:comm}, $D^2U$, has a special structure of rank-one modification of the scalar kernel $\displaystyle \frac{U'(r)}{r}$. This enables us to treat the flocking behavior of  \eqref{eq:AT} for a larger class of purely attractive potentials.
  \begin{assumption}[{\bf Purely attractive potentials}]\label{ass:attractive}
 There exists  constant $0< a < A$ such that 
 \begin{equation}\label{eq:attractive}
a\langle r\rangle^{-\beta} \leq \frac{\wU'(r)}{r} \leq A, \qquad 0\leq \beta\leq 1.
\end{equation}
 \end{assumption}

Our  result is expressed in terms of fluctuations of the anticipated energy 
 \[
 \delcE(t) := \frac{1}{2N}\sum_{i} |\bv_i-\widebar{\bv}|^2 + \frac{1}{2N^2}\sum_{i,j} U(|\bx^\tau_i-\bx^\tau_j|).
 \]
\begin{theorem}[{\bf Anticipation dynamics \eqref{eq:AT} with attractive potential}]\label{thm:thm2}
Consider the anticipation dynamics \eqref{eq:AT} with bounded potential \eqref{eq:bounded}. Assume that $U$ is purely attractive with a fat tail decay of order $\beta$, 
\[
a\langle r\rangle^{-\beta} \le \frac{\wU'(r)}{r} \leq A,  \qquad \beta<\frac{1}{3}. 
\]
Then there is  sub-exponential decay of the anticipated energy fluctuations
\begin{equation}\label{eq:Edecay1}
\delcE(t) \le Ce^{-t^{1-\lamta}},\qquad \lamta = \frac{2\beta}{1-\beta}.
\end{equation}
It follows that for large time, the anticipation dynamics   concentrates in space with global velocity alignment at sub-exponential rate,
\begin{equation}
|\bx_i(t)-\widebar{\bx}(t)|\rightarrow 0,\quad |\bv_i(t)-\widebar{\bv}_0|\rightarrow 0
\end{equation}
\end{theorem}

\begin{remark}
This result is surprising if one interprets \eqref{eq:AT} in its equivalent matrix formulation \eqref{eq:UF}, since attractive potentials do \emph{not} necessarily induce communication matrix $\Phi=D^2U$ which is positive definite. In particular, the corresponding  `regular' (instantaneous) energy $E(t)$ referred to in corollary \ref{cor:convex} is not necessarily decreasing; only the \emph{anticipated} energy is. 
\end{remark}

The proof of Theorem \ref{thm:thm2}, carried out in section \ref{sec:attraction}, involves two main ingredients.\newline
{\bf (i)}.  First, we derive an a priori uniform bound on the maximal spread of \emph{anticipated} positions $|\bx^\tau_i(t)|$, 
\begin{equation}\label{eq:bound}
 \max_i |\bx_i^\tau(t)|  \le C_\infty\brt^{\frac{1}{2-2\beta}}, \qquad 0\leq \beta < 1.
\end{equation}
{\bf (ii)}. A second main ingredient for  the proof of Theorem \ref{thm:thm2} is  based on the 
 energy dissipation \eqref{energy}. The key step  here is to relate the enstrophy
 in \eqref{energy},
  \begin{equation}\label{eq:key}
 \frac{\tau}{N}\sum_i|\dot{\bv}_i|^2=\frac{\tau}{N}\sum_i\Big|\frac{1}{N}\sum_j  c_{ij}(\bx_i^\tau-\bx_j^\tau)\Big|^2, \qquad c_{ij}=\frac{U'(|\bx_i^\tau-\bx_j^\tau|)}{|\bx_i^\tau-\bx_j^\tau|},
 \end{equation}
  to the fluctuations of the (expected) \emph{positions}. This is done by the following proposition, interesting for its own sake, which deals  with the  local vs. global means of arbitrary $\bz_j\in {\mathbb R}^d$.

\begin{lemma}[{\bf Local and global means are comparable}]\label{lem:mean}
Fix $0<\lambda  \le\Lambda $ and weights $\cof_{ij}$
\[
0< {\lambda} \leq \cof_{ij} \leq {\Lambda}.
\]
Then, there exists a constant $C=C(\lambda,\Lambda)\lesssim \frac{\Lambda^2}{\lambda^4}$ (which otherwise is independent of the $\cof_{ij}$'s and $N$) such that for arbitrary  $\bz_j \in {\mathbb R}^d$, with average $\widebar{\bz}:=\frac{1}{N}\sum_j \bz_j$, there holds
\begin{equation}\label{eq:means}
\frac{1}{N}\sum_i \left|\bz_i-\widebar{\bz}\right|^2 \leq \frac{C(\lambda,\Lambda)}{N} \sum_i \Big|\frac{1}{N}\sum_j \cof_{ij}(\bz_i-\bz_j)\Big|^2 , \qquad  C(\Lambda,\lambda) \lesssim \frac{\Lambda^2}{\lambda^4}.
\end{equation}
\end{lemma}
\begin{remark}[{\bf Why a lemma on means?}]\label{rem:why} The last bound \eqref{eq:means} implies --- and in fact is equivalent up to scaling, to the statement about the local means induced by  weights $\theta_{ij}$
\[
\frac{\lambda}{N} \leq \theta_{ij} \leq \frac{\Lambda}{N}, \qquad \sum_j \theta_{ij}=1.
\]
If $\displaystyle \widebar{\bz}_i(\theta):=\sum_j \theta_{ij}\bz_j$ are the local means,  then \eqref{eq:means} with $\cof_{ij}=N\theta_{ij}$ implies
\begin{equation}\label{eq:why}
\frac{1}{N}\sum_i |\bz_i-\widebar{\bz}|^2 \leq \frac{C(\lambda,\Lambda)}{N} \sum_i |\bz_i-{\widebar{\bz}}_i(\theta)|^2, \qquad  C(\Lambda,\lambda) \lesssim \frac{\Lambda^2}{\lambda^4}.
\end{equation}
Thus, the deviation from the local means is comparable to the deviation from the global mean.
\end{remark}

Applying \eqref{eq:means} to \eqref{eq:key} with the given bounds \eqref{eq:attractive},\eqref{eq:bound}, yields
\begin{equation}\label{eq:enspos}
\begin{split}
\frac{\rd}{\rd{t}}\cE(t)& =-\frac{\tau}{N}\sum_i|\dot{\bv}_i|^2 \\
& \lesssim -\frac{\tau}{N}\frac{A^2}{a^4}\Big(\max_{i,j} \braket{\bx^\tau_i-\bx^\tau_j}\Big)^{-4\beta} \sum_i |\bx^\tau_i|^2
\lesssim -\frac{\tau}{2N^2}\brt^{-\frac{2\beta}{1-\beta}} \sum_{i,j} |\bx^\tau_i-\bx^\tau_j|^2.
\end{split}
\end{equation}
Observe that in this case, the enstrophy of the anticipated energy is bounded by the fluctuations of the anticipated positions (compared with velocity fluctuations  in the `regular' energy decay \eqref{energy1}). We close the last  bound by   hypocoercivity argument carried out in section \ref{sec:attraction} which leads to the sub-exponential  decay \eqref{eq:Edecay1}. 

\subsection{Anticipation dynamics with attractive-repulsive potential}
For attractive-repulsive potentials, the large time behavior of \eqref{eq:AT} is significantly more complicated, due to the following two reasons:
\newline
$\bullet$ The topography of the total potential energy $\frac{1}{2N^2}\sum_{i,j} U(|\bx_i-\bx_j|)$ which includes  multiple local minima with different geometric configurations could be  very complicated, see e.g., \cite{LRC2000,DCBC2006,CDP2009,KSUB2011,BCLR2013,CHM2014,Se2015,CCP2017} and the references therein.\newline
$\bullet$ It is numerically observed in \cite{GTLW2017} that the decay of $E(t)$ is of order ${\mathcal O}(t^{-1})$. Therefore one cannot hope that a sub-exponential  energy dissipation rate $\dot{E}(t) \lesssim -\brt^{-\lamta}E(t)$ or its hypocoercivity counterpart to hold. 

Here we focus on the second difficulty, and give a first rigorous result in this direction.
\begin{theorem}[{\bf Anticipation with repulsive-attractive potential}]\label{thm:thm3}
Consider the  2D anticipated dynamics \eqref{eq:AT} of $N=2$ agents subject to   repulsive-attractive potential which has a  local minimum at $r=r_0>0$ where $\wU''(r_0) = a >0$.
Then there exists a constant $\epsilon>0$, such that if the initial data is small enough,
\begin{equation}\label{eq:small_enough}
\big||\bx_1(0)-\bx_2(0)|-r_0\big|^2 + |\bv_1(0)-\bv_2(0)|^2 \le \epsilon,
\end{equation}
then the solution to \eqref{eq:AT} satisfies  the following algebraic decay:
\begin{equation}
\big||\bx_1(t)-\bx_2(t)|-r_0\big| \le C\brt^{-1}\ln^{1/2}\braket{1+t},\quad |\bv_1(t)-\bv_2(t)| \le C\brt^{-1/2}.
\end{equation}
\end{theorem}
The proof, based on  nonlinear hypocoercivity argument for the anticipated energy is carried out in section \ref{sec:ra}.
\begin{remark}
The detailed description of the dynamics outlined in the proof, reveals that the radial component of the velocity, $v_r\lesssim \brt^{-1}\ln^{1/2}\braket{1+t}$, decays faster than its tangential part, $v_\theta \lesssim \brt^{-1/2}$. Therefore, although the dynamics of \eqref{eq2}  can be complicated at the beginning, it will finally settles as a circulation around the equilibrium, provided  the initial data is small enough.
\end{remark}

\subsection{Anticipation hydrodynamics} The large crowd (hydro-)dynamics associated with  \eqref{eq:AT} is described by density and momentum $(\rho,\rho\bu)$ governed by\footnote{Under a simplifying assumption of a  mono-kinetic closure.}   
\begin{equation}\label{eq:hydro}
\left\{\begin{split}
 \rho_t + \nabla_\bx\cdot (\rho \bu) &= 0 \\
 (\rho\bu)_t + \nabla_\bx\cdot (\rho\bu\otimes \bu) & =   - \int \nabla U(|\bx^\tau - \by^\tau|) \rd\rho(\by), \quad \bx^\tau := \bx + \tau \bu(t,\bx).
\end{split}\right.
\end{equation}
The large-time flocking behavior of \eqref{eq:hydro} is studied in terms of lemma \ref{lem:Mean} --- a continuum version of the discrete lemma  of means proved in section \ref{sec:means}. That is, we obtain a sub-linear time bound on the spread of $\text{supp}\,\rho(t,\cdot)$, which in turn is used to control the enstrophy of the anticipated energy. In section \ref{sec:hydro} we outline the proof of our last main result, which states that \emph{if} \eqref{eq:hydro} admits a global smooth solution then such smooth solution must flock, in agreement with the general paradigm for Cucker-Smale dynamics discussed in \cite{TT2014,HeT2017}.

\begin{theorem}[{\bf Anticipation hydrodynamics: smooth solutions must flock}]\label{thm:thm4}
Let $(\rho,\bu)$ be a smooth solution of the anticipation hydrodynamics \eqref{eq:hydro} with  an attractive potential subject to  a fat tail decay, \eqref{eq:attractive}, of order $\beta<\frac{1}{3}$. Then  there is  sub-exponential decay of the anticipated energy fluctuations
\begin{equation}\label{eq:Edecay2}
\int\int \left(\frac{1}{2m_0}|\bu(\bx)-\widebar{\bu}_0|^2+  U(|\bx^\tau-\by^\tau|)\right) \rd\rho(\bx)\rd\rho(\by)\le Ce^{-t^{1-\lamta}}, \quad \lamta=\frac{2\beta}{1-\beta}<1.
\end{equation}
It follows that there is large time flocking, with sub-exponential alignment
\[
|\bu(t,\bx)-\widebar{\bu}_0|^2\rd{\rho(\bx)} \stackrel{t \rightarrow \infty}{\longrightarrow} 0, \qquad \widebar{\bu}_0=\frac{1}{m_0}\int (\rho\bu)_0(\bx) \rd{\bx}, \ \ m_0=\int \rho_0(\bx)\rd{\bx}.
 \]
\end{theorem}
In proposition \ref{prop:1Dexist}  we verify the existence of global smooth solution (and hence flocking) of  the 1D  system, \eqref{eq:hydro},  provided 
the threshold condition, $u'_0(x) \geq -C(\tau,m_0,a)$ holds, for a proper negative constant depending on $\tau, m_0$ and the minimal convexity $a=\min U''>0$.

\section{A priori $L^\infty$ bounds for confining potentials}\label{sec:uniform}
In this section we prove the uniform bounds asserted in \eqref{eq:bound1} and \eqref{eq:bound}, corresponding to the anticipation dynamics  in \eqref{eq:UF} and, respectively, \eqref{eq:AT}. Due to the momentum conservation \eqref{momentum}, we may assume without loss of generality that in both cases $\widebar{\bx}(t)=\widebar{\bv}(t)\equiv0$. This will always be assumed in the rest of this paper.

We recall  the dynamics of \eqref{eq:UF} assumes that $U$ lies in the class of convex potentials, \eqref{eq:convex},  and the dynamics of \eqref{eq:AT} assumes a larger class of attractive potentials, \eqref{eq:attractive}. In fact, here we prove uniform bounds under a more general setup of \emph{confining  potentials}.
 \begin{assumption}[{\bf Confining potentials}]\label{ass:confining}
 There exists  constant $a>0, L\geq0$ such that 
 \begin{equation}\label{eq:confining}
\wU(r) \geq a\left(\langle r\rangle^{2-\beta}-L\right), \qquad 0\leq \beta\leq 1.
\end{equation}
 \end{assumption}
 \noindent
Observe that in particular, attractive potentials are confining\footnote{Thus, we have the increasing hierarchy of three classes of potentials --- convex, attractive and confining potentials.}, 
\begin{equation}\label{eq:attconf}
\wU(r)=\int_0^r \wU'(s)\rd{s} \geq \int_0^r a\langle s\rangle^{-\beta}s\rd{s}=
\frac{a}{2-\beta}\left(\langle r\rangle^{2-\beta}-1\right).
\end{equation}
The class of confining potentials is much larger, however, and it includes repulsive-attractive potentials (discussed in section \ref{sec:ra} below). In particular, a confining $U$ needs not be positive.

 \begin{lemma}[{\bf Uniform bound on positions for \ref{eq:UF} system}]\label{lem:convex}
Consider the anticipation dynamics \eqref{eq:UF} with bounded positive communication matrix $0\leq \Phi\leq \phi_+$, and bounded confining   potential \eqref{eq:bounded},\eqref{eq:confining}.
Then the solution $\{(\bx_i(t),\bv_i(t))\}$ satisfies the a priori estimate
\begin{equation}\label{eq:cX1}
 \max_i |\bx_i(t)|  \le C_\infty\brt^{\frac{2}{4-3\beta}}, \quad \max_i|\bv_i(t)| \leq C_\infty\brt^{\frac{2-\beta}{4-3\beta}}, \qquad  0\leq \beta \leq 1.
\end{equation}
\end{lemma}
\begin{remark} Note that we require a positive communication matrix $\Phi$ but otherwise, we do not insist on any fat tail condition \eqref{eq:positive}.
\end{remark}
\begin{proof}
 Our proof is based on the technique introduced in \cite[\S2.2]{ST2019}, in which we prove uniform bounds in terms of the \emph{particle energy}\footnote{In fact $E_i$ is not a proper particle energy, since $\sum_i E_i \ne N E$ (the pairwise potential is counted twice). However, it is  the ratio of the kinetic energy and potential energy in (\ref{cEi}) which is essential, as one would like to eliminate all the positive terms with indices $i$ in (\ref{cEi}), in order to avoid exponential growth of $E_i$.}
\begin{equation}\label{cEi}
E_i(t) := \frac{1}{2}|\bv_i|^2 + \frac{1}{N}\sum_j U(|\bx_i-\bx_j|).
\end{equation}
We start by relating the local energy to the position of particle $i$: using \eqref{eq:confining}
followed by Jensen inequality for the convex mapping\footnote{\label{fn:jensen}$(\brr^{2-\beta})'' = -\beta(2-\beta) r^2\brr^{-2-\beta} + (2-\beta)\brr^{-\beta} = (2-\beta)\big((1-\beta) r^2 + 1\big)\brr^{-2-\beta}>0$ for $\beta\leq 1$.} $\bx\mapsto \langle \bx \rangle^{2-\beta}$, we find
\[
\frac{E_i(t)}{a} \ge  \frac{1}{N}\sum_j \left(\langle \bx_i-\bx_j\rangle^{2-\beta}-L\right) 
\ge  \big\langle \frac{1}{N}\sum_j (\bx_i-\bx_j) \big\rangle^{2-\beta}-L =  \langle \bx_i\rangle^{2-\beta}-L
\]
It follows that the maximal spread of positions, $\cX(t):=\max_i|\bx_i(t)|$, does not exceed
\begin{equation}\label{cF1}
\cX(t) \le \left(\frac{E_\infty(t)}{a}+L\right)^{\frac{1}{2-\beta}}, \qquad \cX(t)=\max_i |\bx_i(t)|, \quad  E_\infty(t)= \max_i E_i(t).
\end{equation}
Next we bound the energy dissipation rate of each particle. By \eqref{eq:positive} the communication matrices $\Phij$ are non-negative and bounded\footnote{Observe that we do not use the fat tail decay \eqref{eq:positive}.}, $0\leq \Phij\leq \phi_+$,  and since $\sum_j \bv_j=0$, 
\begin{equation}
\begin{split}
\frac{\rd}{\rd{t}}E_i(t) = & \bv_i\cdot\left(-\frac{1}{N}\sum_j \nabla U(|\bx_i-\bx_j|)  + \frac{1}{N}\sum_j \Phij(\bv_j-\bv_i)\right) \\
 \quad & + \frac{1}{N}\sum_j \nabla U(|\bx_i-\bx_j|)\cdot(\bv_i-\bv_j) \\
= & \frac{1}{N}\sum_j \Phij(\bv_j-\bv_i)\cdot\bv_i - \frac{1}{N}\sum_j \nabla U(|\bx_i-\bx_j|)\cdot\bv_j \\
= & -\frac{1}{2N}\sum_j \Phij\bv_i\cdot\bv_i-\frac{1}{2N}\sum_j \Phij(\bv_j-\bv_i)\cdot(\bv_j-\bv_i) \\
& + \frac{1}{2N}\sum_j \Phij\bv_j\cdot\bv_j- \frac{1}{N}\sum_j \Big(\nabla U(|\bx_i-\bx_j|)-\nabla U(|\bx_i|)\Big)\cdot\bv_j \\
\le & {\phi_+ E(0)} + \sqrt{2E(0)}\left(\frac{1}{N}\sum_j \big|\nabla U(|\bx_i-\bx_j|)-\nabla U(|\bx_i|)\big|^2\right)^{1/2}. 
\end{split}
\end{equation}
To bound the sum on the right, we use the fact that $D^2U$ is bounded, \eqref{eq:bounded}, followed by \eqref{cF1} to find, 
\begin{equation}\label{xitau2}
\begin{split}
 \frac{1}{N}\sum_{j} |&\nabla U(|\bx_i-\bx_j|)-\nabla U(|\bx_i|)|^2 \\
  & \le  \sup_{\bx}|D^2U(|\bx|)|^2 \frac{1}{N}\sum_{j} |\bx_j|^2 
\le   \frac{A^2}{N}\sum_{j} |\bx_j|^2 
\\
& =  \frac{A^2}{2N^2}\sum_{i,j} |\bx_i-\bx_j|^2 
\le 
 2^\beta A^2\max_i|\bx_i|^\beta\times \frac{1}{N^2}\sum_{i,j} |\bx_i-\bx_j|^{2-\beta}
\\
& \le   2^{\beta} A^2 \cX^{\beta} \frac{1}{2N^2}\sum_{i,j} |\bx_i-\bx_j|^{2-\beta} 
\le  2^{\beta}A^2 \cX^{\beta} \frac{1}{2N^2}\sum_{i,j} \left(\frac{U(|\bx_i-\bx_j|)}{a}+L\right) \\
& \le  2^{\beta}A^2 \cX^{\beta}\left(\frac{E(0)}{a}+\frac{L}{2}\right).
\end{split}
\end{equation}

Therefore
\[
\begin{split}
\frac{\rd}{\rd{t}}E_i(t) \le &{\phi_+ E(0)} + \sqrt{2E(0)}\left(2^{\beta}A^2\cX^{\beta}\Big(\frac{E(0)}{a}+\frac{L}{2}\Big)\right)^{1/2} \\
\le & {\phi_+ E(0)} + \sqrt{2E(0)}\left(2^{\beta}A^2\Big(\frac{E_\infty}{a}+L\Big)^{\frac{\beta}{2-\beta}}\Big(\frac{E(0)}{a}+\frac{L}{2}\Big)\right)^{1/2} \\
\end{split}
\]
and taking maximum among all $i$'s,
\begin{equation}
\begin{split}
\frac{\rd}{\rd{t}}E_\infty(t) \le  & {\phi_+ E(0)} + \sqrt{2E(0)}\left(2^{\beta}A^2\Big(\frac{E_\infty(t)}{a}+L\Big)^{\frac{\beta}{2-\beta}}\Big(\frac{E(0)}{a}+\frac{L}{2}\Big)\right)^{1/2}.
\end{split}
\end{equation}
Set $f(t) := E_\infty(t) + aL$, then the last inequality tells us 
$f' \le C_1 + C_2 f^\alpha$ with $\alpha:=\frac{\beta}{(4-2\beta)}$, and
since by assumption $\alpha<\nicefrac{1}{2}$, then
\[
f \lesssim \brt^{\frac{1}{1-\alpha}} = C \brt^{\frac{2(2-\beta)}{4-3\beta}},
\]
 which implies the uniform bound on velocities in \eqref{eq:cX1}, 
\[
\max_i|\bv_i(t)| \leq 2\sqrt{E_\infty(t)+aL} \lesssim \brt^{\frac{2-\beta}{4-3\beta}}.
\]
The uniform bound on positions, $\max_i|\bx_i(t)|$, follows in view of  \eqref{cF1}.
\end{proof}

Lemma \ref{lem:convex} applies, in particular, to the anticipation dynamics \eqref{eq:AT} with convex potential, so that $D^2U$ is positive definite. Next, we prove uniform bounds for more general confining $U$'s.

\begin{lemma}[{\bf Uniform bound on anticipated positions}]\label{lem:bound}
Consider the anticipation dynamics \eqref{eq:AT} with bounded confining potential \eqref{eq:bounded},\eqref{eq:confining}.
Then the solution of the anticipation dynamics  \eqref{eq:AT} satisfies the a priori estimate
\begin{equation}\label{eq:cX}
 \max_i |\bx_i^\tau(t)|  \le C_\infty\brt^{\frac{1}{2-2\beta}}, \qquad  0\leq \beta < 1.
\end{equation}
\end{lemma}
\begin{remark}
The a priori bound \eqref{eq:cX} is weaker than lemma \ref{lem:convex} and may not be optimal for $\beta$ close to 1. We do not pursue an improved bound  since it does not provide an increased range of $\beta$'s for which Theorem \ref{thm:thm2} holds. 
\end{remark}
\begin{proof}[Proof of Lemma \ref{lem:bound}]
The key quantity for proving the priori bound \eqref{eq:cX}
is the `anticipated particle energy' in \eqref{eq:AT},
\begin{equation}\label{Ei}
\cE_i(t) := \frac{1}{2}|\bv_i|^2 + \frac{1}{N}\sum_j U(|\bx_i^\tau-\bx_j^\tau|).
\end{equation}
Similar to the previous proof, the confining property of $U$ implies that the diameter of anticipated positions,  $\cXtau(t):=\max_i |\bx_i^\tau(t)|$, does not exceed
\begin{equation}\label{F1}
 \cXtau(t) \le \left(\frac{\cE_\infty(t)}{a}+L\right)^{\frac{1}{2-\beta}}, \qquad \cXtau(t):=\max_i |\bx^\tau_i(t)|, \quad  \cE_\infty(t)= \max_i \cE_i(t). 
\end{equation}
Next we bound the energy dissipation rate of each particle: since $\sum_j (\bv_j+\tau\dot{\bv}_j)=0$, 
\begin{equation}\label{dEi}
\begin{split}
\frac{\rd}{\rd{t}}\cE_i(t) = & \ \bv_i\cdot \dot{\bv}_i + \frac{1}{N}\sum_{j} \nabla U(|\bx_i^\tau-\bx_j^\tau|)\cdot (\bv_i + \tau \dot{\bv}_i - \bv_j - \tau \dot{\bv}_j) \\
= & -\tau |\dot{\bv_i}|^2 - \frac{1}{N}\sum_{j} \nabla U(|\bx_i^\tau-\bx_j^\tau|)\cdot(\bv_j+ \tau \dot{\bv}_j) \\
= & -\tau |\dot{\bv_i}|^2 - \frac{1}{N}\sum_{j} \big(\nabla U(|\bx_i^\tau-\bx_j^\tau|)-\nabla U(|\bx_i^\tau|)\big)\cdot(\bv_j+ \tau \dot{\bv}_j).
\end{split}
\end{equation}
As before, the boundedness of $D^2U$ followed by \eqref{F1} to find  imply
\begin{equation}\label{xitau2}
\frac{1}{N}\sum_{j} \big|\nabla U(|\bx_i^\tau-\bx_j^\tau|)-\nabla U(|\bx_i^\tau|)\big|^2 \le   2^{\beta}A^2 \cXtau^{\beta}\left(\frac{\cE(0)}{a}+\frac{L}{2}\right).
\end{equation}

Inserting \eqref{xitau2} into the RHS of \eqref{dEi} and adding the energy-enstrophy  balance \eqref{energy} we find\footnote{Note that a confining potential need not be positive yet $U\geq -aL$ and hence  $1/2N\sum_j |\bv_j|^2 \le \cE(0)+aL$.}
\[
\begin{split}
\frac{\rd}{\rd{t}}&(\cE(t) +  \cE_i(t)) \\
 & \le  - \frac{\tau}{N}\sum_j |\dot{\bv}_j|^2 -  \tau|\dot{\bv}_i|^2 +  \frac{c}{N}\sum_j |\bv_j|^2 + \frac{c\tau^2}{N}\sum_j |\dot{\bv}_j|^2 \\
  & \ \ \ \ +  \frac{1}{4cN}\sum_{j} \big|\nabla U(|\bx_i^\tau-\bx_j^\tau|)-\nabla U(|\bx_i^\tau|)\big|^2 \\
& \le  - \frac{\tau(1- c \tau)}{N}\sum_j |\dot{\bv}_j|^2 -  \tau|\dot{\bv}_i|^2 + 2 c \big(\cE(0) + aL\big)  +  \frac{1}{4c}2^{\beta}A^2 \cXtau^{\beta}\left(\frac{\cE(0)}{a}+\frac{L}{2}\right) \\
& \le    \frac{2}{\tau} \big(\cE(0) +aL\big) +   \frac{\tau}{4}2^{\beta}A^2 \cXtau^{\beta}\left(\frac{\cE(0)}{a}+\frac{L}{2}\right), \qquad \mbox{(taking $c=1/\tau$)}.
\end{split}
\]
 By taking maximum among all $i$,
\[
\begin{split}
\frac{\rd}{\rd{t}}(\cE(t) +  \cE_\infty(t)) \le &   \frac{2}{\tau} \big(\cE(0) +aL\big)  +   \frac{\tau}{4}2^{\beta}A^2 \cXtau^{\beta}\left(\frac{\cE(0)}{a}+\frac{L}{2}\right) \\
\le &   \frac{2}{\tau} \big(\cE(0) +aL\big)  +  \frac{\tau}{4}2^{\beta}A^2\left(\frac{\cE_\infty(t)}{a}+L\right)^{\frac{\beta}{2-\beta}}\left(\frac{\cE(0)}{a}+\frac{L}{2}\right) \\
\end{split}
\]
The last inequality tells us that $f(t) := \cE(t) +  \cE_\infty(t) + aL$ satisfies
$f' \le C_1 + C_2 f^\alpha$ with $\alpha:=\frac{\beta}{2-\beta}$. Since by assumption $\alpha<1$, then
\[
f \lesssim \brt^{\frac{1}{1-\alpha}} = C \brt^{\frac{2-\beta}{2-2\beta}},
\]
and the uniform bound \eqref{eq:cX}  follows in view of \eqref{F1}.
\end{proof}

\section{Anticipation with convex potentials and positive communication kernels}\label{sec:convex}
Equipped with the uniform bound \eqref{eq:cX}, we turn to  prove Theorem \ref{thm:thm1} by hypocoercivity argument. In \cite{ST2019} we use  hypocoercivity to prove the flocking  with \emph{quadratic potentials}. Here, we make a judicious use of the assumed fat tail conditions, \eqref{eq:positive},\eqref{eq:convex}, to extend these arguments for general  convex potentials.

\begin{proof}[Proof of Theorem \ref{thm:thm1}]
We introduce the \emph{modified energy}, $\widehat{E}(t)$, by adding a multiple of the cross term $\nicefrac{1}{N}\sum_i \bx_i\cdot\bv_i$, 
\[
\widehat{E}(t):=E(t) + \frac{\epsilon(t)}{N}\sum_i\bx_i(t)\cdot\bv_i(t).
\]
We claim that with a proper choice $\epsilon(t)$, the modified energy  is positive definite. Indeed,
 the convex (hence attractive) potential satisfies the pointwise bound \eqref{eq:attconf}, and together with the uniform bound \eqref{eq:bound1}  they imply
\[
\begin{split}
|\epsilon(t)\frac{1}{N}\sum_i \bx_i \cdot\bv_i| \le & \frac{1}{4N}\sum_i |\bv_i|^2 +  \frac{\epsilon^2(t)}{N}\sum_i |\bx_i|^2 
= \frac{1}{4N}\sum_i |\bv_i|^2 +  \frac{\epsilon^2(t)}{2N^2}\sum_{i,j} |\bx_i-\bx_j|^2
\\
\le & \frac{1}{4N}\sum_i |\bv_i|^2 + \epsilon^2(t) (2\cX(t))^{\beta}\frac{1}{2N^2}\sum_{i,j} \left(\frac{U(|\bx_i-\bx_j|)}{a}+1\right) \\
\le & \frac{1}{4N}\sum_i |\bv_i|^2 + \epsilon^2(t)(2C_\infty)^{\beta} \brt^{\frac{2\beta}{4-3\beta}}\frac{1}{2N^2}\sum_{i,j} \left(\frac{U(|\bx_i-\bx_j|)}{a}+1\right).
\end{split}
\]
Therefore it suffices to choose 
\begin{equation}\label{epsilon2}
\epsilon(t) = \epsilon_0\brt^{-\alpha}, \qquad  \ \alpha > \frac{\beta}{4-3\beta},
\end{equation}
with  small enough $\epsilon_0>0$  and \emph{any} $\alpha > \frac{\beta}{4-3\beta}$ which is to be determined later, to guarantee 
$|\nicefrac{\epsilon(t)}{N}\sum_i\bx_i\cdot\bv_i| \le {E(t)}/{2}$, hence the positivity of $\widehat{E}(t)\geq E(t)/2>0$.

Next, we turn to verify the coercivity of $\widehat{E}(t)$. 
First notice that Lemma \ref{lem:convex} implies the following $L^\infty$ bound on $\bx_i^{\tau_{ij}}$:
\[
|\bx_i^{\tau_{ij}}| \le |\bx_i| + \tau|\bv_i| \le (1+\tau)C_\infty\brt^{\frac{2}{4-3\beta}}
\]
This together with the assumed fat-tails of  $\Phi$ and $D^2U$, imply their lower-bounds: by \eqref{eq:positive} $\Phij$ are bounded from below by,
\begin{equation}\label{eq:phiminus}
\Phij\geq \phi_-(t)  := c\brt^{-\frac{2\gamma}{4-3\beta}},
\end{equation}
and integrating \eqref{eq:convex} $U''(r) \geq a\langle r\rangle^{-\beta}$ twice, implies $U$ has the lower bound \eqref{eq:attconf}
\begin{equation}\label{psi}
U(|\bx_i-\bx_j|) \geq c|\bx_i-\bx_j|^2 \langle |\bx_i-\bx_j| \rangle^{-\beta}
\geq |\bx_i-\bx_j|^2\psi_-(t), \quad \psi_-(t):=c\brt^{-\frac{2\beta}{4-3\beta}}.
\end{equation}

Now, we turn to conduct hypocoercivity argument based on the energy estimate \eqref{energy1}. To this end, we append to $E(t)$, a proper multiple of the cross term $\sum\bx_i\cdot\bv_i$, consult e.g., \cite{ST2019,DS2019}. Using the symmetry of $\Phij$, the  time derivative of this   cross term is given by
\begin{equation}\label{eq:cross}
\begin{split}
 \frac{\rd}{\rd{t}}\frac{1}{N}\sum_i &\bx_i\cdot\bv_i \\
= & \ \frac{1}{N}\sum_i |\bv_i|^2 + \frac{1}{N}\sum_i \bx_i\cdot\left( -\frac{1}{N}\sum_j \nabla U(|\bx_i-\bx_j|) +  \frac{1}{N}\sum_j \Phij(\bv_j-\bv_i) \right) \\
= & \ \frac{1}{2N^2}\sum_{i,j} |\bv_i-\bv_j|^2 - \frac{1}{2N^2}\sum_{i,j} (\bx_i-\bx_j)\cdot \nabla U(|\bx_i-\bx_j|) \\
& +  \ \frac{1}{2N^2}\sum_{i,j}  \Phij(\bv_j-\bv_i)\cdot(\bx_i-\bx_j).
\end{split}
\end{equation}
We prepare the following three bounds. Noticing that since $U$ is convex $U'(r)$ is increasing, hence $\displaystyle U(r) = \int_0^r U'(s) \rd{s} \leq rU'(r)$ implies
\begin{subequations}\label{eqs:cross}
\begin{equation}\label{eq:cross1}
\begin{split}
\frac{1}{2N^2}\sum_{i,j} (\bx_i-\bx_j)\cdot \nabla U(|\bx_i-\bx_j|) 
&= \frac{1}{2N^2}\sum_{i,j}U'(|\bx_i-\bx_j|) |\bx_i-\bx_j|  \\
& \ge  \frac{1}{2N^2}\sum_{i,j} U(|\bx_i-\bx_j|). 
\end{split}
\end{equation}
Using the weighted Cauchy-Schwarz twice --- weighted by the positive definite $0 < \Phij\le \phi_+$, and then by the yet-to-be determined $\kapta(t)>0$,
\begin{equation}\label{eq:cross2}
\begin{split}
\left|  \frac{1}{2N^2}\sum_{i,j}\right. &  \left.\Phij(\bv_j-\bv_i)\cdot(\bx_i-\bx_j) \right| \\ & \le  \frac{\kapta}{4N^2}\sum_{i,j} \Phij(\bv_i-\bv_j)\cdot(\bv_i-\bv_j) + \frac{1}{4\kapta N^2}\sum_{i,j} \Phij(\bx_i-\bx_j)\cdot(\bx_i-\bx_j)\\
& \le  \frac{\kapta(t)}{4N^2}\sum_{i,j} \Phij(\bv_i-\bv_j)\cdot(\bv_i-\bv_j) + \frac{\phi_+}{4\kapta(t) N^2}\sum_{i,j} |\bx_i-\bx_j|^2
\end{split}
\end{equation}
Recall that with the choice of $\epsilon(t)=\epsilon_0\brt^{-\alpha}$ in \eqref{epsilon2}, we have 
$\displaystyle  |\dot{\epsilon}(t)| \le \alpha \frac{\epsilon(t)}{\brt}$. We have the final bound
\begin{equation}\label{eq:cross3}
\begin{split}
\Big|\frac{\dot{\epsilon}(t)}{N}\sum_i\bx_i\cdot \bv_i \Big|
& \leq  \left|\dot{\epsilon}(t)\right|\frac{1}{2\delta(t) N^2}\sum_{i,j} |\bx_i|^2
+ \left|\dot{\epsilon}(t)\right| \frac{\delta(t)}{2N^2}\sum_{i,j} |\bv_i|^2\\
 & \leq \frac{\alpha }{2\delta(t) \brt}\frac{\epsilon(t)}{2N^2}\sum_{i,j} |\bx_i-\bx_j|^2 + \frac{\alpha \delta(t)}{2 \brt}\frac{\epsilon(t)}{2N^2}\sum_{i,j} |\bv_i-\bv_j|^2 
\end{split}
\end{equation}
\end{subequations}
Adding \eqref{eq:cross}  to the energy decay \eqref{energy1} we find that 
the dissipation rate of the modified energy $\widehat{E}(t):=E(t) + \nicefrac{\epsilon(t)}{N}\sum_i\bx_i(t)\cdot\bv_i(t))$ does not exceed, in view of   \eqref{eqs:cross}
\begin{equation}\label{eq:dissipative}
\begin{split}
\frac{\rd}{\rd{t}}\widehat{E}(t) 
\le & \left( -\tau+\frac{\kapta(t)}{2}\epsilon(t)\right)\frac{1}{2N^2}\sum_i \Phij(\bv_i-\bv_j)\cdot(\bv_i-\bv_j) \\
& \ \ + \left(\frac{\phi_+}{2\kapta(t)}\epsilon(t) + \frac{\alpha}{2\delta(t)\brt}\epsilon(t)\right)\frac{1}{2N^2}\sum_{i,j} |\bx_i-\bx_j|^2\\
 & \ \ \ \ + \left( \epsilon(t) +\frac{\alpha\delta(t)}{2\brt}\epsilon(t)\right)\frac{1}{2N^2}\sum_{i,j} |\bv_i-\bv_j|^2\\
 & \ \ \ \ \ \ -\epsilon(t)\frac{1}{2N^2}\sum_{i,j} U(|\bx_i-\bx_j|) \\
& =: I + II + III +IV.
\end{split}
\end{equation}
To complete the (hypo-)coercivity argument,  we  guarantee the   terms on the right of \eqref{eq:dissipative} are negative. To this end, set $\kapta(t)=\tau/\epsilon(t)$ so the first pre-factor$\leq -\nicefrac{\tau}{2}$ hence
\[
I \leq -\frac{\tau}{2}\phi_-(t) \frac{1}{2N^2}\sum_{i,j} |\bv_i-\bv_j|^2=-\frac{\tau}{2}\phi_-(t) \frac{1}{N}\sum_i |\bv_i|^2, \qquad \kapta(t)=\frac{\tau}{\epsilon(t)}.
\]
Next, we set $\displaystyle \delta(t)=\frac{\delta_0}{\epsilon(t)\brt}$ so that the  second pre-factor $\displaystyle\leq \left(\frac{\phi_+}{\tau}+\frac{\alpha}{2\delta_0}\right)\epsilon^2(t)$, hence the second term does not exceed, in view of \eqref{psi}
\[
II \leq  \left(\frac{\phi_+}{\tau}+\frac{\alpha}{2\delta_0}\right)\frac{\epsilon^2(t)}{\psi_-(t)} \frac{1}{2N^2}\sum_{i,j} U(|\bx_i-\bx_j|), \qquad \delta(t)=\frac{\delta_0}{\epsilon(t)\brt}.
\]
With these choices of $\kapta$ and $\delta$, the third term does not exceed
\[
III \le   \left( \epsilon(t) +\frac{\alpha\delta_0}{2\brt^2}\right)\frac{1}{2N^2}\sum_{i,j} |\bv_i-\bv_j|^2 =\left( \epsilon(t) +\frac{\alpha\delta_0}{2\brt^2}\right)\frac{1}{N}\sum_{i} |\bv_i|^2
\]
We conclude
\begin{equation}\label{eq:hypo}
\begin{split}
\frac{\rd}{\rd{t}}\widehat{E}(t) & \le
\left(-\frac{\tau}{2}\phi_-(t)+\epsilon(t)+\frac{\alpha\delta_0}{2\brt^2}\right) \frac{1}{N} \sum_i |\bv_i|^2\\
& \ \ +\left(- \epsilon(t) +\Big(\frac{\phi_+}{\tau}+\frac{\alpha}{2\delta_0}\Big)\frac{\epsilon^2(t)}{\psi_-(t)}\right)  \frac{1}{2N^2}\sum_{i,j} U(|\bx_i-\bx_j|).
\end{split}
\end{equation}
Now set $\displaystyle\alpha\geq \frac{2\gamma}{4-3\beta}$ so that $\phi_-(t)$ decays no faster than  $\epsilon(t)$; moreover, $\phi_-(t)$ decays no faster than  $\brt^{-2}$ since $6\beta+2\gamma \leq 8$, and hence, with small enough $\epsilon_0,\delta_0>0$, the first pre-factor on the right of \eqref{eq:hypo} does not exceed $-{\tau}\phi_-(t)/4$. Next, let $\displaystyle\alpha\geq \frac{2\beta}{4-3\beta}$ so that 
$\epsilon(t)/\psi_-(t)$ is bounded: hence, with small enough $\epsilon_0 \ll \delta_0$, the second pre-factor on the right of \eqref{eq:hypo} does not exceed $-\epsilon(t)/2$. We conclude
\[
\frac{\rd}{\rd{t}}\widehat{E}(t) \lesssim 
-\frac{\phi_-(t)}{N}\sum_i |\bv_i|^2
-\frac{\epsilon(t)}{2N^2}\sum_{i,j} U(|\bx_i-\bx_j|)
\lesssim -\brt^{-\lamta} \widehat{E}(t), \quad {\lamta}=\frac{2\max\{\beta,\gamma\}}{4-3\beta}.
\]
This implies the sub-exponential decay of $\widehat{E}$, and thus that of the comparable $E$.
\end{proof}

\subsection{Flocking of matrix-based Cucker-Smale dynamics}\label{subsec:CS}

The Cucker-Smale  model  \cite{CS2007a,CS2007b},
\begin{equation}\label{eq:CS}
\left\{\begin{split}
& \dot{\bx}_i=\bv_i \\
& \dot{\bv}_i = \frac{\tau}{N}\sum_{j=1}^N \Phij(\bv_j-\bv_i),  
\end{split}\right.
\end{equation}
is a special case of \eqref{eq:UF} with no external potential  $U=0$, which formally corresponds to $\beta=0$, in which case theorem \ref{thm:thm1} would yield  flocking for $\gamma <\nicefrac{1}{2}$. Here we justify these formalities and prove the velocity alignment  of \eqref{eq:CS} (no spatial concentration effect, however), under a slightly larger threshold.
\begin{proposition}[{\bf Alignment of} \eqref{eq:CS}  {\bf model with positive kernels}]\label{prop:CS}
Consider the Cucker-Smale dynamics \eqref{eq:CS} with symmetric  matrix kernel $\Phi$ satisfying 
\[
\phi_-\braket{\bx_i-\bx_j}^{-\gamma} \leq \Phi(\bx_i,\bx_j) \leq \phi_+, \qquad 0\leq \gamma < \nicefrac{2}{3}.
\]
Then there is  sub-exponential decay of the  energy fluctuations
\begin{equation}\label{eq:CSfracdecay}
\delE(t) \le Ce^{-t^{1-\lamta}}, \qquad \lamta=\frac{3\gamma}{2}, \qquad \delE(t):= \frac{1}{2N}\sum_{i} |\bv_i-\widebar{\bv}|^2.
\end{equation}
It follows that there is a flock formation around the mean $\widebar{\bx}(t)$ with large time velocity alignment at sub-exponential rate:
\begin{equation}
\bv_i(t)\rightarrow \widebar{\bv}_0,\quad \bx_i(t) - \widebar{\bx}(t) \rightarrow \bx_i^\infty, \qquad \widebar{\bx}(t):=\widebar{\bx}_0+t\widebar{\bv}_0,
\end{equation}
for some constants $\bx_i^\infty$.
\end{proposition}
The proof  is similar but follows a slightly different strategy from that of theorem \ref{thm:thm1}: we start by a priori estimate for the particle energy $E_i$, and then proceed to controlling the position $\braket{\bx_i}$, which in turn gives enough energy dissipation.

\begin{proof}
Define the particle energy
\begin{equation}
E_i(t) := \frac{1}{2}|\bv_i|^2,\quad E_\infty(t) = \max_i E_i(t).
\end{equation}
Observe that it satisfies 
\begin{equation}
\begin{split}
\frac{\rd}{\rd{t}} E_i(t) = & \frac{1}{N}\sum_{j=1}^N  \Phij(\bv_j-\bv_i)\cdot \bv_i  \\
= & -\frac{1}{2N}\sum_{j=1}^N  \Phij(\bv_j-\bv_i)\cdot (\bv_j-\bv_i)-\frac{1}{2N}\sum_{j=1}^N \ \Phij\bv_i\cdot \bv_i+ \frac{1}{2N}\sum_{j=1}^N  \Phij\bv_j\cdot \bv_j\\
\le & -\phi_-(t)\frac{|\bv_i|^2}{2} + \phi_+E(t),
\end{split}
\end{equation}
where $\phi_-(t)$, the lower-bound  of the symmetric $\Phi(\bx_i,\bx_j)$ which is given, in view of  \eqref{eq:positive},
\begin{equation}
\phi_-(t) = a 2^{-\gamma}\cX^{-\gamma}(t) ,\quad \cX(t) := \max_i |\bx_i(t)|.
\end{equation}
Taking $i$ as the particle with the largest $E_i$, then 
\begin{equation}\label{Einfty_1}
\frac{\rd}{\rd{t}} E_\infty(t)  \le -\phi_-(t)E_\infty(t) + \phi_+E(t) \le -\phi_-(t)E_\infty(t) + \phi_+E(0).
\end{equation}
This  implies
\begin{equation}
E_\infty(t) \le E_\infty(0)+\phi_+E(0)t. 
\end{equation}

Next, we notice that
\begin{equation}
\frac{\rd}{\rd{t}} \cX(t) \le \max_i|\bv_i| \le \sqrt{2E_\infty(t)} \le \sqrt{2(E_\infty(0)+\phi_+E(0)t)}. 
\end{equation}
This yields $\cX(t) \le C\brt^{3/2}$, and in view of the fat tail \eqref{eq:positive}, we end with  the lower bound
\[
\phi_-(t) \ge c\brt^{-\frac{3\gamma}{2}}.
\]
Therefore the energy dissipation \eqref{energy1} gives
\begin{equation}
\frac{\rd}{\rd{t}} E(t) \le -\phi_-(t) E(t) \le -c\brt^{-\frac{3\gamma}{2}}E(t),
\end{equation}
which implies the sub-exponential decay \eqref{eq:CSfracdecay},
$\displaystyle 
E(t) \le E(0)e^{-c\brt^{1-{\lamta}}}$ with ${\lamta}=\frac{3\gamma}{2}<1$.

 Equipped with this sub-exponential decay of $E(t)$, we revisit \eqref{Einfty_1}: this time it implies
\begin{equation}
\begin{split}
E_\infty(t) &\le  e^{-\Phi_-(t)} E_\infty(0) + \phi_+\int_0^t e^{-\Phi_-(t-s)} E(s) \rd{s}\\
 & \le  C\brt e^{-c\brt^{1-{\lamta}}}, \qquad  \Phi_-(t) := \int_0^t \phi_-(s)\rd{s} \ge c\brt^{1-{\lamta}}
\end{split}
\end{equation}
This shows the sub-exponential decay of the kinetic energy of each agent, $E_\infty(t)$, independent oh $N$, $|\bv_i(t)-\widebar{\bv}_0| \rightarrow 0$. As a result,
$\ds \bx_i(t) = \bx_i(0) + \int_0^t \bv_i(s)\rd{s}$, 
converges as $t\rightarrow\infty$ since the last integral converges absolutely in view of $|\bv_i(t)| \le \sqrt{2E_\infty(t)}$.
\end{proof}

\section{Local vs. global weighted means}\label{sec:means}
In this section we prove  lemma \ref{lem:mean}   about discrete  means, which in turn will be used in proving the  hypocoercivity  of  the discrete  anticipation dynamics \eqref{eq:AT}. We also treat the corresponding continuum lemma of means in lemma \ref{lem:Mean}, which is utilized in the  hypocoercivity of the   hydrodynamic anticipation model \eqref{eq:hydro}. 

We begin with the proof of the Lemma of means \ref{lem:mean}:
\begin{proof}[Proof of Lemma \ref{lem:mean}]
We first treat the scalar setup, in which case we may assume, without loss of generality that  the $z_i$'s are rearranged in a decreasing order,  $z_1\ge z_2 \ge \cdots \ge z_N$, and have a zero mean $\sum_j z_j=0$. 
Let $i_0$ be the smallest index $i$ such that 
\begin{equation}\label{eq:igti0}
\frac{1}{N}\sum_{j=1}^{i-1} z_j \ge \frac{\lambda}{2(\Lambda-\lambda)}z_i.
\end{equation}
Noticing that if $i_+$ is the maximal index of  the positive entries, $z_{i\leq i_+}\geq0$, then \eqref{eq:igti0} clearly holds for $i >  i_+$ (where LHS $> 0 >$ RHS), hence $i_0\leq i_+$, and since  LHS is increasing (for $i \leq i_+$) and RHS is decreasing,  see figure \ref{fig:graph} below, \eqref{eq:igti0} holds for all $i\geq i_0$
\begin{equation}\label{eq:ave2}
\frac{1}{N}\sum_{j=1}^{i-1} z_j \ge \frac{\lambda}{2(\Lambda-\lambda)}z_i, \quad i\geq i_0.
\end{equation}
For $i<i_0$ we have  $z_i\geq0$, hence
\begin{equation}\label{eq:ave1}
\begin{split}
\frac{1}{N}\sum_j \cof_{ij}(z_i-z_j) &=  \frac{1}{N}\sum_{j=1}^{i} \cof_{ij}(z_i-z_j) + \frac{1}{N}\sum_{j=i+1}^{N} \cof_{ij}(z_i-z_j)   \\
 &\ge   \frac{\Lambda}{N}\sum_{j=1}^{i} (z_i-z_j) + \frac{\lambda}{N}\sum_{j=i+1}^{N} (z_i-z_j) \\
 & = -\frac{\Lambda}{N}\sum_{j=1}^{i} z_j + \frac{\lambda}{N}\sum_{j=1}^{i} z_j +\frac{\Lambda}{N}\sum_{j=1}^{i} z_i + \frac{\lambda}{N}\sum_{j=i+1}^{N} z_i \\
 &\ge  -\frac{\Lambda-\lambda}{N}\sum_{j=1}^{i-1} z_j + \lambda z_i, \qquad i<i_0,
\end{split}
\end{equation}
and therefore, by the minimality of $i_0$ in \eqref{eq:ave2}
\[
\frac{1}{N}\sum_j \cof_{ij}(z_i-z_j) \geq -\frac{\Lambda-\lambda}{2(\Lambda-\lambda)}\lambda z_i +\lambda z_i = \frac{\lambda}{2}z_i \geq 0, \qquad i < i_0.
\]
It follows that
\begin{equation}\label{eq:ave3}
\frac{1}{N} \sum_{i=1}^{i_0-1} \Big|\frac{1}{N}\sum_j \cof_{ij}(z_i-z_j)\Big|^2 \ge \frac{\lambda^2}{4}\frac{1}{N}\sum_{i=1}^{i_0-1} z_i^2.
\end{equation}
Else, for $i\geq i_0$, \eqref{eq:ave2} implies
\[
z_i\leq z_{i_0} \le \frac{2(\Lambda-\lambda)}{\lambda}\frac{1}{N}\sum_{j=1}^{i_0-1} z_j, \quad i\geq i_0. 
\]
It follows that for all positive entries, $0\leq z_i\leq z_{i_0}$,
\begin{equation}\label{eq:ave4}
\frac{1}{N}\sum_{i=i_0}^{i_+}z_i^2 \le z^2_{i_0} \leq \frac{4(\Lambda-\lambda)^2}{\lambda^2}\frac{1}{N^2}\left(\sum_{j=1}^{i_0-1} z_j\right)^2 \leq \frac{4(\Lambda-\lambda)^2}{\lambda^2}\frac{1}{N}\sum_{j=1}^{i_0-1} z_j^2
\end{equation}
Therefore, by \eqref{eq:ave4},\eqref{eq:ave3},
\begin{equation}\label{eq:pos}
\begin{split}
\frac{1}{N}\sum_{z_i\ge 0}z_i^2 &= \frac{1}{N}\sum_{i=1}^{i_0-1}z_i^2
+\frac{1}{N}\sum_{i=i_0}^{i_+}z_i^2 \\
& \leq \left(1+ \frac{4(\Lambda-\lambda)^2}{\lambda^2}\right)\frac{1}{N}\sum_{j=1}^{i_0-1} z_j^2 \\
& \le \frac{4}{\lambda^2}\left(1+4\Big(\frac{\Lambda}{\lambda} -1\Big)^2\right)\frac{1}{N} \sum_i \Big|\frac{1}{N}\sum_j \cof_{ij}(z_i-z_j)\Big|^2.
\end{split}
\end{equation}
\begin{figure}
	\includegraphics[height=4.5in]{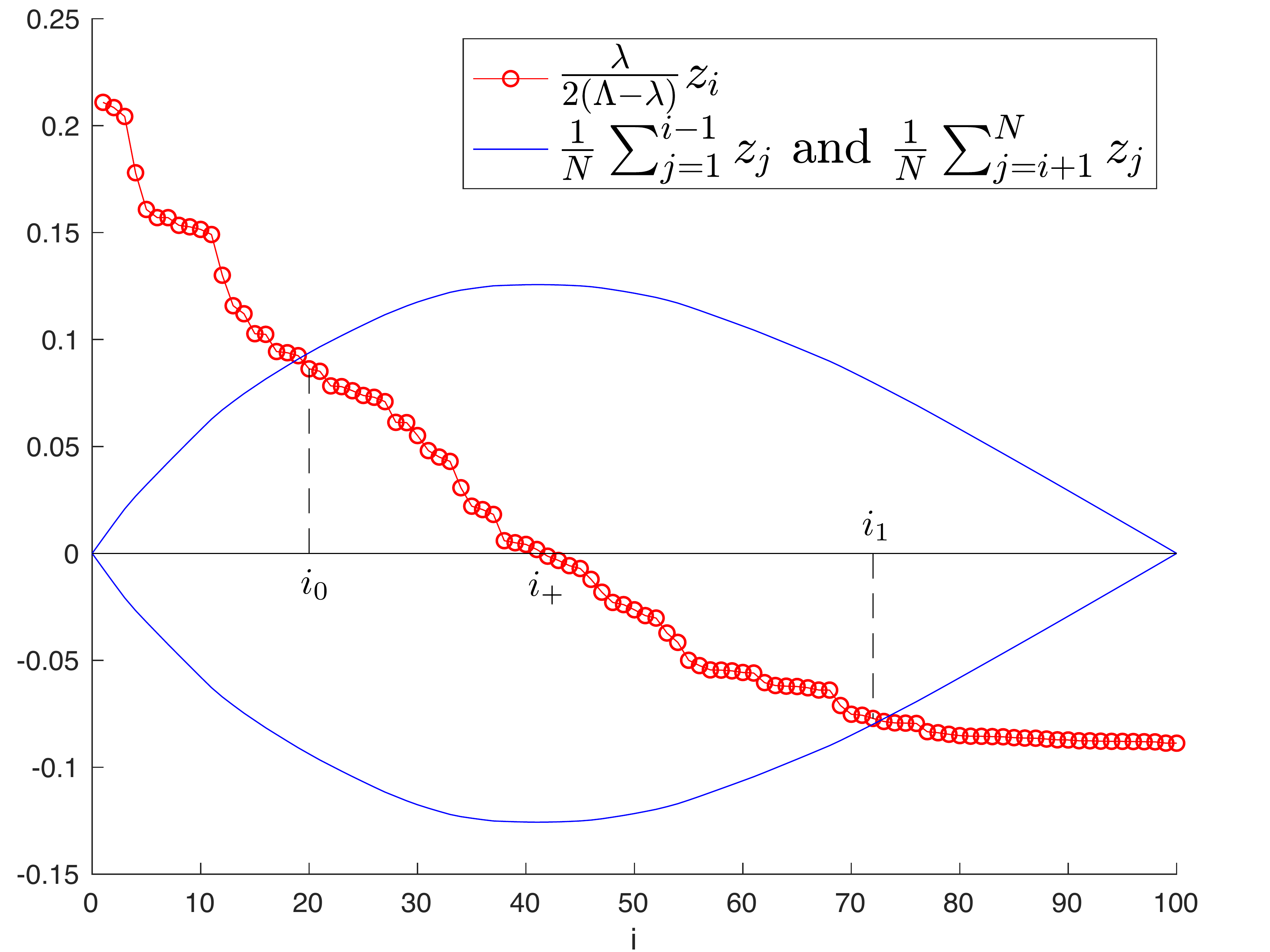}
	\caption{Comparison of means 
	}\label{fig:graph}
\end{figure}
Now apply \eqref{eq:pos} to $z_i$ replaced by $-z_i$, to find  the same upper-bound on the negative entries
\begin{equation}\label{eq:neg}
\begin{split}
\frac{1}{N}\sum_{z_i\le 0}z_i^2  
\le \frac{4}{\lambda^2}\left(1+4\Big(\frac{\Lambda}{\lambda}-1\Big)^2\right)\frac{1}{N} \sum_i \Big|\frac{1}{N}\sum_j \cof_{ij}(z_i-z_j)\Big|^2.
\end{split}
\end{equation}
The scalar result follows from \eqref{eq:pos},\eqref{eq:neg}.
For the $d$-dimensional case, notice that 
\[
\begin{split}
\sum_i |\bz_i|^2 &= \sum_{k=1}^d \sum_i |z_i^k|^2 \\
  \sum_i \Big|\frac{1}{N}\sum_j \cof_{ij}(\bz_i-\bz_j) \Big|^2 
&=    \sum_{k=1}^d  \sum_i\Big|\frac{1}{N}\sum_j \cof_{ij}(z^k_i-z^k_j) \Big|^2
\end{split}
\]
where superscript stands for component. Therefore the conclusion follows by applying the scalar result to the components of $\bz_i=\{z_i^k\}_k$ for each fixed $k$, ending with the same constant $C(\Lambda,\lambda)$ which is  independent of $d$.
\end{proof}

Next, we extend the result from the discrete framework to the continuum.
\begin{lemma}[{\bf Local and global means are comparable}]\label{lem:Mean}
Let $(\Omega,\mathcal{F},\mu)$ be a probability measure, and $\bX:\Omega\rightarrow\mathbb{R}^d$ be a random variable with mean $\ds \widebar{\bX}:= \int \bX(\omp)\rd{\mu(\omp)}$ and finite second moment, $\ds \int |\bX(\omp)|^2\rd{\mu(\omp)}<\infty$. Then  there exists a constant $C(\Lambda,\lambda) \lesssim \Lambda^2 \lambda^{-4}$, such that for any    $c=c(\om,\omp): \Omega\times\Omega \mapsto {\mathbb R}$ satisfying 
\[
0<\lambda\le c(\om,\omp) \le \Lambda,
\] 
there holds
\[
\int |\bX(\om)-\widebar{\bX}|^2 \rd{\mu(\om)} \le C(\lambda,\Lambda)\int \left| \int  c(\om,\omp) (\bX(\om)-\bX(\omp))\rd{\mu(\omp)} \right|^2 \rd{\mu(\om)}, \quad C(\Lambda,\lambda) \lesssim \frac{\Lambda^2}{\lambda^{4}}.
\]
\end{lemma}
Observe that  the particular case of $\ds \rd{\mu}=\frac{1}{N}\sum_j\delta(\bx-\bz_j)$ recovers the discrete case of lemma \ref{lem:mean}.

\begin{proof}
We first prove the 1D case, for which the map $\bX$,  denoted by $X$, may assume  a zero mean, $\widebar{X}=0$, without loss of generality. Take $\om$ with $X(\om):=x\ge 0$, then
\[
\begin{split}
\int c(\om,\omp)(x&-X(\omp))\rd{\mu(\omp)} \\
= & -\int_{\omp: X(\omp)>x} \hspace*{-0.4cm}c(\om,\omp)(X(\omp)-x)\rd{\mu(\omp)} + \int_{\omp: X(\omp)\le x} \hspace*{-0.4cm}c(\om,\omp)(x-X(\omp))\rd{\mu(\omp)} \\
 \ge & -\Lambda\int_{\omp: X(\omp)>x} (X(\omp)-x)\rd{\mu(\omp)} + \lambda\int_{\omp: X(\omp)\le x} (x-X(\omp))\rd{\mu(\omp)} \\
 = & -(\Lambda-\lambda)\int_{\omp: X(\omp)>x} (X(\omp)-x)\rd{\mu(\omp)} + \lambda x\\
 \ge & -(\Lambda-\lambda)\int_{\omp: X(\omp)>x} X(\omp)\rd{\mu(\omp)} + \lambda x \\
\end{split}
\]
Let  
\begin{equation}
x_0 := \sup\left\{x: Y(x) \ge 0\right\}, \qquad Y(x) =\int_{\omp: X(\omp)>x}X(\omp)\rd{\mu(\omp)} - \frac{\lambda}{2(\Lambda-\lambda)}x. 
\end{equation}
Since $\lim_{x\rightarrow\infty}Y(x)=-\infty$ and $Y(0)\ge 0$, $x_0$ is finite and non-negative. It is clear that $Y(x)$ is decreasing and right-continuous. Therefore $Y(x)\ge 0$ for any $x< x_0$, and $Y(x)\le 0$ for any $x\ge x_0$.  

If $x\ge x_0$, then
\begin{equation}
\begin{split}
\int c(\om,\omp)(x-X(\omp))\rd{\mu(\omp)}  \ge -(\Lambda-\lambda)\frac{\lambda}{2(\Lambda-\lambda)}x + \lambda x =  \frac{\lambda}{2}x
\end{split}
\end{equation}
Thus taking square and integrating in $\om$ with $x=X(\om)\ge x_0 \ge 0$ gives
\begin{equation}\label{meanc1}
\int_{\om: X(\om)\ge x_0} \left(\int c(\om,\omp)(X(\om)-X(\omp))\rd{\mu(\omp)}\right)^2\rd{\mu(\om)} \ge   \frac{\lambda^2}{4}\int_{\om: X(\om)\ge x_0} \hspace*{-0.9cm}X^2(\om) \rd{\mu(\om)}.
\end{equation}

Then we claim that the above integral on $\{\om: X(\om)\ge x_0\}$ is enough to get the conclusion.  Notice that for any $\epsilon>0$, one has $Y(x_0-\epsilon)\ge 0$, i.e.,
\begin{equation}
x_0-\epsilon \le \frac{2(\Lambda-\lambda)}{\lambda}\int_{\om: X(\om)> x_0-\epsilon} X(\omp)\rd{\mu(\omp)}
\end{equation}
and therefore
\[
\begin{split}
(x_0-\epsilon)^2 & \le  \frac{4(\Lambda-\lambda)^2}{\lambda^2}\left(\int_{\om: X(\om)> x_0-\epsilon} X(\omp)\rd{\mu(\omp)}\right)^2 \\
& \le  \frac{4(\Lambda-\lambda)^2}{\lambda^2}\left(\int_{\om: X(\om)> x_0-\epsilon} X(\omp)^2\rd{\mu(\omp)}\right)\left(\int_{\om: X(\om)> x_0-\epsilon} \rd{\mu(\omp)}\right) \\
&\le \frac{4(\Lambda-\lambda)^2}{\lambda^2}\int_{\om: X(\om)> x_0-\epsilon} X(\omp)^2\rd{\mu(\omp)} 
\end{split}
\]
Taking $\epsilon\rightarrow 0$, noticing that the RHS integral domain $\{\om: X(\om)> x_0-\epsilon\}$ converges to $\{\om: X(\om)\ge x_0\}$, we get
\begin{equation}\label{meanc2}
x_0^2 \le \frac{4(\Lambda-\lambda)^2}{\lambda^2}\int_{\om: X(\om)\ge x_0} X(\omp)^2\rd{\mu(\omp)}
\end{equation}
Thus, using (\ref{meanc2}) and (\ref{meanc1}) we find
\[
\begin{split}
\int_{\omp: X(\omp)\ge 0} \hspace*{-0.6cm}X(\omp)^2&\rd{\mu(\omp)} =  \int_{\omp: 0\le X(\omp)< x_0} X(\omp)^2\rd{\mu(\omp)} + \int_{\omp: X(\omp)\ge x_0} X(\omp)^2\rd{\mu(\omp)} \\
\le & \left(\frac{4(\Lambda-\lambda)^2}{\lambda^2}+1\right)\int_{\omp: X(\omp)\ge x_0} X(\omp)^2\rd{\mu(\omp)}
 \\
 \le & \left(\frac{4(\Lambda-\lambda)^2}{\lambda^2}+1\right)\frac{4}{\lambda^2}\int_{[x_0,\infty)} \left(\int c(\om,\omp)(X(\om)-X(\omp))\rd{\mu(\omp)}\right)^2\rd{\mu(\om)}.
\end{split}
\]
Apply the last bound with $X(\cdot)$ replaced by $-X(\cdot)$ to find  that the $\ds \int_{\omp: X(\omp)\le 0} \hspace*{-0.6cm}X(\omp)^2\rd{\mu(\omp)} $ satisfies the same bound  on the right, which completes the scalar part of the proof.
For the $d$-dimensional case with $\bX = (X_1,\dots,X_d)$, notice that 
\[
\int |\bX(\omp)|^2 \rd{\mu(\omp)} = \sum_{k=1}^d \int |X_k(\omp)|^2 \rd{\mu(\omp)}
\]
and similarly,
\[
\int \left| \int  c(\om,\omp) (\bX(\om)-\bX(\omp))\rd{\mu(\omp)} \right|^2 \rd{\mu(\om)} = \sum_{k=1}^d \int \left| \int  c(\om,\omp) (X_k(\om)-X_k(\omp))\rd{\mu(\omp)} \right|^2 \rd{\mu(\om)}
\]
Applying the 1D result to the random variable $X_k$ gives 
\stepcounter{equation}
\begin{equation}\tag*{(\theequation)$_{k}$}\label{eq:mk}
\int |X_k(\omp)|^2 \rd{\mu(\omp)} \le C(\lambda,\Lambda)\int \left| \int  c(\om,\omp) (X_k(\om)-X_k(\omp))\rd{\mu(\omp)} \right|^2 \rd{\mu(\om)}.
\end{equation}
Summing \ref{eq:mk} recovers the desired result  with the constant $C(\Lambda,\lambda)$ independent of $d$.
\end{proof}

\section{Anticipation dynamics with attractive potentials}\label{sec:attraction}

In this section we prove the flocking behavior of \eqref{eq:AT} asserted in Theorem \ref{thm:thm2}. Here, we treat the larger class of attractive potentials, thus extending the case of convex potentials of Theorem \ref{thm:thm1}. 
The starting point is the anticipated energy balance \eqref{energy}
\[
\frac{\rd}{\rd{t}}\cE(t) = -  \frac{\tau}{N}\sum_{i} | \dot{\bv}_i|^2.
\]

\begin{remark} We note in passing that the first-order system 
\[
\dot{\bx}_i = -\frac{1}{N}\sum_{j=1}^N \nabla U(|\bx_i-\bx_j|),
\]
 satisfies an   energy estimate, reminiscent of the energy-enstrophy balance in  the anticipation dynamics \eqref{eq:AT}, 
\[
\begin{split}
\frac{\rd}{\rd{t}} \frac{1}{2N^2}\sum_{i,j}U(|\bx_i-\bx_j|) = -\frac{1}{N}\sum_i|\dot{\bx}_i|^2.
\end{split}
\]
\end{remark} 

\begin{proof}[Proof of Theorem \ref{thm:thm2}]
 We aim to conduct a hypocoercivity argument to complement the anticipated energy estimate  \eqref{energy}.  
  To this end, we use the `anticipated'  cross term,
\begin{equation}
\begin{split}
\frac{\rd}{\rd{t}}(-\frac{1}{N}\sum_i \bx_i^\tau\cdot\bv_i) 
= &  \frac{1}{N}\sum_i \left( -(\bv_i + \tau\dot{\bv}_i)\cdot\bv_i - \bx_i^\tau\cdot\dot{\bv}_i \right) \\
\le &  \frac{1}{N}\sum_i \left( -|\bv_i|^2 + \tau(\frac{\tau}{2}|\dot{\bv}_i|^2 + \frac{1}{2\tau}|\bv_i|^2) + \frac{1}{2}(|\bx_i^\tau|^2 + |\dot{\bv}_i|^2) \right) \\
\le & -\frac{1}{2} \frac{1}{N}\sum_i |\bv_i|^2 + \frac{1}{2}\frac{1}{N}\sum_i|\bx_i^\tau|^2 + \frac{\tau^2+1}{2}\frac{1}{N}\sum_i|\dot{\bv}_i|^2
\end{split}
\end{equation}
Consider the modified anticipated energy $\widehat{\cE}(t):= \cE(t)-\epsilon(t)\frac{1}{N}\sum_i \bx_i^\tau\cdot\bv_i$, where $\epsilon(t)>0$ is small, decreasing, and is yet to be chosen. 
We first need to guarantee that this modified energy is positive definite, and in fact --- comparable to $\cE(t)$,
\begin{equation}\label{pd1}
\big|\epsilon(t)\frac{1}{N}\sum_i \bx_i^\tau\cdot\bv_i\big| \le \frac{\cE(t)}{2}
\end{equation}
Indeed, notice that
\[
\begin{split}
|\epsilon(t)\frac{1}{N}\sum_i \bx_i^\tau\cdot\bv_i| \le & \frac{1}{4N}\sum_i |\bv_i|^2 + \epsilon^2(t)\frac{1}{N}\sum_i |\bx_i^\tau|^2 \\
\le & \frac{1}{4N}\sum_i |\bv_i|^2 + \epsilon(t)^2 (2\cXtau)^{\beta}\frac{1}{2N^2}\sum_{i,j} \left(\frac{U(|\bx_i^\tau-\bx_j^\tau|)}{a}+1\right) \\
\le & \frac{1}{4N}\sum_i |\bv_i|^2 + \epsilon^2(t)
 (2C_\infty)^{\beta} \brt^{\frac{\beta}{2-2\beta}}\frac{1}{2N^2}\sum_{i,j} \left(\frac{U(|\bx_i^\tau-\bx_j^\tau|)}{a}+1\right).
\end{split}
\]
The second inequality is obtained similarly to (\ref{xitau2}), and the third inequality uses Lemma \ref{lem:bound}. Therefore it suffices to choose
\begin{equation}\label{epsilon1}
\epsilon(t) = \epsilon_0(10+t)^{-\alpha},\quad \alpha \ge \frac{\beta}{4-4\beta}
\end{equation}
with small enough $\epsilon_0$  to guarantee \eqref{pd1}. 

We now turn to verify the (hypo-)coercivity of $\widehat{\cE}(t)$,
\begin{equation}\label{hypo1}
\begin{split}
\frac{\rd}{\rd{t}} \Big(\cE(t) &-\epsilon(t)\frac{1}{N}\sum_i \bx_i^\tau\cdot\bv_i\Big) \\
 \le &  -  \frac{\tau}{N}\sum_{i} | \dot{\bv}_i|^2 -\frac{\epsilon(t)}{2} \frac{1}{N}\sum_i |\bv_i|^2 \\
 & \ \ + \epsilon(t)\left(\frac{1}{2}\frac{1}{N}\sum_i|\bx_i^\tau|^2 + \frac{\tau^2+1}{2}\frac{1}{N}\sum_i|\dot{\bv}_i|^2\right) + |\dot{\epsilon}(t)|\frac{1}{N}\sum_i |\bx_i^\tau\cdot\bv_i| \\
 \le &  - \left(\tau- \epsilon(t)\frac{\tau^2+1}{2}\right) \frac{1}{N}\sum_{i} | \dot{\bv}_i|^2 \\
 & \ \ -\frac{\epsilon(t)- |\dot{\epsilon}(t)|}{2}   \frac{1}{N}\sum_i |\bv_i|^2 + \frac{\epsilon(t)+ |\dot{\epsilon}(t)|}{2}\frac{1}{N}\sum_i|\bx_i^\tau|^2 
\end{split}
\end{equation}
The first pre--factor  on the right of \eqref{hypo1} $\leq -\tau/2$ for small enough $\epsilon_0$. The second pre-factor is  negative  since 
\[
|\dot{\epsilon}(t)| = \alpha\epsilon_0 (10+t)^{-\alpha-1} \le \frac{\alpha}{10}\epsilon(t).
\]
It remains to control  the last term on the right of \eqref{hypo1}. To this end we recall that $U$ is assumed attractive, $U'(r)/r \geq \langle r\rangle ^{-\beta}$, hence, by Lemma \ref{lem:bound}
\[
A \geq \frac{\wU'(r^\tau_{ij})}{r^\tau_{ij}} \ge a\langle r^\tau_{ij}\rangle^{-\beta} \ge c\brt^{-\frac{\beta}{2-2\beta}}, \qquad r^\tau_{ij}=|\bx_i^\tau-\bx_j^\tau|.
\]
We now  invoke Lemma \ref{lem:mean}: it implies
\begin{equation}
\begin{split}
\frac{1}{N}\sum_{i} | \dot{\bv}_i|^2 =  \frac{1}{N} \sum_i \left|\frac{1}{N}\sum_j \frac{\wU'(r^\tau_{ij})}{r^\tau_{ij}} (\bx_i^\tau-\bx_j^\tau)\right|^2
 \ge  c \brt^{-\lamta}\frac{1}{N}\sum_i|\bx_i^\tau|^2 ,\quad \lamta = \frac{2\beta}{1-\beta}
\end{split}
\end{equation}

Therefore, the last term on the right of \eqref{hypo1} does not exceed $\lesssim \epsilon(t)\brt^\lamta \frac{1}{N}\sum_i |\dot{\bv}_i|^2$ and choosing $\epsilon(t)$ as in (\ref{epsilon1}) with $\alpha=\lamta$ yields 
\begin{equation}
\frac{\rd}{\rd{t}} \widehat{\cE}(t) \le - \frac{\tau}{4} \frac{1}{N}\sum_{i} | \dot{\bv}_i|^2 -\frac{\epsilon_0(10+t)^{-\lamta}}{4}   \frac{1}{N}\sum_i |\bv_i|^2.
\end{equation}
As before, since $U$ is bounded, it has at most quadratic growth,
\[
\begin{split}
\frac{1}{2N^2}\sum_{i,j} U(|\bx_i^\tau-\bx_j^\tau|) \le & A\frac{1}{2N^2}\sum_{i,j} |\bx_i^\tau-\bx_j^\tau|^2 
\le  C\brt^{\frac{2\beta}{1-\beta}}\frac{1}{N}\sum_{i} | \dot{\bv}_i|^2 
=  C\brt^{\lamta}\frac{1}{N}\sum_{i} | \dot{\bv}_i|^2,
\end{split}
\]
and we conclude the sub-exponential decay
\begin{equation}
\frac{\rd}{\rd{t}}\widehat{\cE}(t) \le -c\brt^{-\lamta} \widehat{\cE}(t)\ \leadsto \ \widehat{\cE}(t) \le C e^{-t^{1-\lamta}},
\end{equation}
which implies the same decay rate of $\cE(t)$.
\end{proof}

\section{Anticipation  dynamics with repulsive-attractive potential}\label{sec:ra}
In this section we prove Theorem \ref{thm:thm3}. The assumption $\sum_i\bx_i=\sum_i\bv_i=0$ amounts to saying that $\bx:=\bx_1=-\bx_2$, $\bv:=\bv_1=-\bv_2$. Replacing $U(|\bx|)$ by $U(2|\bx|)$ and $r_0$ by $r_0/2$, \eqref{eq:AT} becomes
\begin{equation}\label{eq2}
\left\{\begin{split}
& \dot{\bx}=\bv\\
& \dot{\bv} = - \nabla U(|\bx^\tau|)
\end{split}\right.
\end{equation}
where $\wU(r)$ has a local minimum at $r=r_0>0$ with $\wU''(r_0) = a >0$.

We use polar coordinates
\begin{equation}
\left\{\begin{split}
& x_1^\tau = r\cos\theta \\
& x_2^\tau = r\sin\theta \\
\end{split}\right.
\end{equation}
and
\begin{equation}
\left\{\begin{split}
& v_r = v_1\cos\theta+v_2\sin\theta \\
& v_\theta = -v_1\sin\theta+v_2\cos\theta \\
\end{split}\right.
\end{equation}
Then (\ref{eq2}) becomes
\begin{equation}\label{eqkr}
\left\{\begin{split}
& \dot{r} = v_r - \tau \wU'(r) \\
& \dot{\theta}=\frac{v_\theta}{r} \\
& \dot{v}_r= -\wU'(r) + \frac{v_\theta^2}{r}\\
& \dot{v}_\theta= \frac{-v_rv_\theta}{r} \\
\end{split}\right.
\end{equation}
We will focus on perturbative solutions near $r=r_0, v_r=v_\theta=0$. Write $r:=r_0+\kr$, and there hold the approximations
\begin{equation}
\wU(r) \approx \frac{a}{2}\kr^2,\quad\wU'(r) \approx a \kr ,\quad\wU''(r) \approx a
\end{equation}
Observe that our assumed initial configuration in \eqref{eq:small_enough}
implies, and in fact is equivalent to the assumption of smallness on the  \emph{anticipated} energy, $\cE(0) \le 2(1+\tau) \epsilon$. Theorem \ref{thm:thm3} is a consequence of the following proposition on the polar system \eqref{eqkr},
\begin{proposition}[{\bf polar coordinates}]
There exists a constant $\epsilon>0$, such that if the initial data is small enough:
\begin{equation}
\cE_0 := \left(\wU(r) + \frac{1}{2}v_r^2 + \frac{1}{2}v_\theta^2\right)_{|{t=0}} \le \epsilon
\end{equation}
then the solution to \eqref{eqkr} decays to zero at the following algebraic rates:
\begin{equation}
\kr \le C\brt^{-1}\ln^{1/2}\brt,\quad v_r\le C\brt^{-1}\ln^{1/2}\brt,\quad v_\theta \le C\brt^{-1/2}
\end{equation}
\end{proposition}

\begin{proof}
Fix $0<\zeta\le \min\{\frac{r_0}{2},1\}$ as a small number such that
\begin{equation}
\frac{a}{2}\le \wU''(r) \le 2a,\quad \forall \kr^2 \le \zeta,
\end{equation}
and as a result,
\begin{equation}\label{zeta}
\frac{a}{2}|\kr| \le |\wU'(r)| \le 2a|\kr|,\quad \frac{a}{4}\kr^2 \le \wU(r) \le a\kr^2,\quad  \forall \kr^2 \le \zeta.
\end{equation}
We start from the energy estimate for the anticipated energy $\displaystyle \cE(t):= \wU(r) + \frac{1}{2}v_r^2 + \frac{1}{2}v_\theta^2$
\[
\begin{split}
\frac{\rd}{\rd{t}}\cE(t) &  =   \wU'(r) \cdot (v_r - \tau \wU'(r)) + v_r\cdot \left(-\wU'(r) + \frac{v_\theta^2}{r}\right) + v_\theta\cdot \frac{-v_rv_\theta}{r}=  -\tau  \wU'(r)^2
\end{split}
\]
Therefore, for any positive $\ds \epsilon\le \frac{a}{4}\zeta$ to be chosen later, if $\cE_0 \le \epsilon$, then 
\begin{equation}\label{small}
\kr^2 \le \frac{4}{a}\wU(r) \le \frac{4}{a}\epsilon< \zeta, \qquad v_r^2 \le \epsilon,
\end{equation}
hold for all time which in turn implies that \eqref{zeta} holds.
Next we consider the cross terms
\begin{equation}\label{eq:rtcross1}
\begin{split}
\frac{\rd}{\rd{t}} (-v_rv_\theta^2) =  - \left( -\wU'(r) + \frac{v_\theta^2}{r}\right)v_\theta^2 - 2v_rv_\theta\frac{-v_rv_\theta}{r} =  - \frac{v_\theta^4}{r} + \wU'(r) v_\theta^2 + 2\frac{v_r^2v_\theta^2}{r},
\end{split}
\end{equation}
and
\begin{equation}\label{eq:rtcross2}
\begin{split}
\frac{\rd}{\rd{t}} (-\wU'(r)v_r) = &  -\wU''(r)v_r\cdot(v_r - \tau \wU'(r)) - \wU'(r)\left(-\wU'(r) + \frac{v_\theta^2}{r}\right) \\
= & -\wU''(r)v_r^2 + \tau \wU''(r)v_r \wU'(r) + \wU'(r)^2 - \wU'(r)\frac{v_\theta^2}{r}
\end{split}
\end{equation}
 We now introduce the modified energy,
\[
\widehat{\cE}(t):= \wU(r) + \frac{1}{2}v_r^2 + \frac{1}{2}v_\theta^2 -c v_rv_\theta^2 - c \wU'(r)v_r,
\]
depending on a small $c>0$ which is yet to be determined.
A straightforward  calculation, based on \eqref{zeta}  shows its decay rate does not exceed
\[
\begin{split}
\frac{\rd}{\rd{t}}\widehat{\cE}(t) &=   -\tau \wU'(r)^2 - \frac{c}{r}v_\theta^4 -c\wU''(r)v_r^2 \\ 
& \ \ \ + c\left(\wU'(r) v_\theta^2 + 2\frac{v_r^2v_\theta^2}{r}\right)  + c\left(\tau \wU''(r)v_r \wU'(r) + \wU'(r)^2 - \wU'(r)\frac{v_\theta^2}{r}\right) \\
& \le  - \tau \frac{a^2}{4}\kr^2 - \frac{c}{2r_0}v_\theta^4 -c\frac{a}{2}v_r^2 \\ 
 & \ \ \ + c\left(2a|\kr| v_\theta^2 + 4\frac{v_r^2v_\theta^2}{r_0}\right)
  + c\left(4\tau a^2  |v_r \kr| + 4a^2 \kr^2 + 4a|\kr| \frac{v_\theta^2}{r_0}\right),
\end{split}
 \]
 and by Cauchy-Schwarz
 \begin{equation}\label{eq:hatE}
 \begin{split}
\frac{\rd}{\rd{t}}\widehat{\cE}(t) & \le  - \tau \frac{a^2}{4}\kr^2 - \frac{c}{2r_0}v_\theta^4 -c\frac{a}{2}v_r^2 \\ 
& \ \ \ + c\left(\frac{1}{\kapta} a\kr^2 + \kapta a v_\theta^4 + \kapta \frac{2}{r_0}v_\theta^4 + \frac{1}{\kapta}\frac{2}{r_0}v_r^4 \right) \\
&  \ \ \ + c\left(2\kapta \tau a^2  v_r^2 + \frac{1}{\kapta}2\tau a^2  \kr^2 + 4a^2 \kr^2 + \frac{1}{\kapta}\frac{2a}{r_0}\kr^2+ \kapta\frac{2a}{r_0}v_\theta^4 \right) \\
& =  - \left(\tau \frac{a^2}{4} - \frac{c}{\kapta}\big(a+2\tau a^2 + 4\kapta a^2 +\frac{2a}{r_0}\big) \right)\kr^2 \\
& \ \ \ - c\left( \frac{1}{2r_0} -\kapta(a+\frac{2}{r_0} + \frac{2a}{r_0})\right)v_\theta^4  - c\left(\frac{a}{2} - \frac{1}{\kapta}\frac{2}{r_0}v_r^2 -  2\kapta\tau a^2\right)v_r^2, 
\end{split}
\end{equation}
with $\kapta\in (0,1)$ which is yet to be determined.
We want to guarantee that the three pre-factors on the right are positive. To this end, we first fix the ratio 
\stepcounter{equation}
\begin{equation}\tag*{(\theequation)$_{1}$}\label{eq:pre1}
\frac{c}{\kapta} = \frac{\tau \frac{a^2}{8}}{a+2\tau a^2 + 4 a^2 +\frac{2a}{r_0}}
\end{equation}
so that the  first pre-factor is lower-bounded by $\tau \frac{a^2}{8}$. Then we choose 
\begin{equation}\tag*{(\theequation)$_{2}$}\label{eq:pre2}
\kapta \le \min\left\{ 1, \frac{\frac{1}{4r_0}}{a+\frac{2}{r_0} + \frac{2a}{r_0}}, \frac{\frac{a}{4}}{2\tau a^2} \right\}
\end{equation}
so that the second pre-factor, the coefficient of $v_\theta^4$, becomes  $\ds \ge\frac{c}{4r_0}$. Finally, the third pre-factor is also positive because (i) a small enough $\kapta$ was chosen in \ref{eq:pre2}, and (ii) a key  aspect in which  $v_r$ can be made small enough to compensate for small $\kapta$, so that  the negative contribution of $\ds - \frac{1}{\kapta}\frac{2}{r_0}v_r^2$ can be absorbed into the rest: indeed, if 
\begin{equation}\tag*{(\theequation)$_{3}$}\label{vr_small}
v_r^2 \le \frac{\frac{a}{8}}{\frac{1}{\kapta}\frac{2}{r_0}} = \frac{ar_0}{16}\kapta
\end{equation}
then 
\[
c\left(\frac{a}{2} - \frac{1}{\kapta}\frac{2}{r_0}v_r^2 -  2\kapta\tau a^2\right) \leq c\left(\frac{a}{2} - \frac{1}{\kapta}\frac{2}{r_0} \frac{\frac{a}{8}}{\frac{1}{\kapta}\frac{2}{r_0}} -  \frac{a}{4}\right) = \frac{ca}{8}.
\]
Therefore,  \eqref{eq:hatE} implies  the decay rate
\begin{equation}\label{eq:Ehatdecay}
\frac{\rd}{\rd{t}}\widehat{\cE}(t) \le -\lamta_1(\kr^2 + v_\theta^4 + v_r^2), \qquad \lamta_1 = \min\left\{\tau \frac{a^2}{8}, \frac{c}{4r_0}, \frac{ca}{8}\right\},
\end{equation}
provided (\ref{zeta}) and \ref{eq:pre1}--\ref{vr_small} are satisfied. 

Moreover, we claim that $\widehat{\cE}$ is comparable to the original anticipated  energy $\cE$. Indeed, if in addition
\stepcounter{equation}
\begin{equation}\tag*{(\theequation)$_{1}$}\label{eq:pre4}
 c\sqrt{\frac{ar_0}{16}\kapta} \le \frac{1}{4}
\end{equation}
holds, then in view of 
 \ref{vr_small}, $\ds c|v_rv_\theta^2| \le \frac{1}{4}v_\theta^2$, and 
 if 
 \begin{equation}\tag*{(\theequation)$_{2}$}\label{eq:pre5}
c\le \min\Big\{\frac{1}{8}, \frac{1}{4a}\Big\},
\end{equation}
holds, then in view of \eqref{zeta},
\[
c|\wU'(r)v_r| \le ca(\kr^2 + v_r^2) \le ca\Big(\frac{4}{a}\wU(r) + v_r^2\Big) \le \frac{1}{2}\Big(\wU(r)+\frac{1}{2}v_r^2\Big).
\]
It follows that 
\begin{equation}\label{eq:comp}
\frac{1}{2}\cE(t) \le \widehat{\cE}(t)\le 2\cE(t),
\end{equation}
provided \ref{eq:pre4}--\ref{eq:pre5} are satisfied.
These last two  conditions  are clearly met for small enough  $\kapta$: recall that the ratio $c/\kapta$ was fixed in \ref{eq:pre1} then 
\stepcounter{equation}
\begin{equation}\tag*{(\theequation)$_{1}$}\label{eq:pre6}
\kapta \le \frac{a+2\tau a^2 + 4 a^2 +\frac{2a}{r_0}}{\tau \frac{a^2}{8}} \min\{ \sqrt{ar_0}, \frac{1}{8}, \frac{1}{4a}\}
\end{equation}
suffices to guarantee \ref{eq:pre4}--\ref{eq:pre5}. 
Thus, we finally choose small enough $\kapta$  to satisfy both \ref{eq:pre2},\ref{eq:pre6}, and small enough 
$\ds \epsilon < \min\big\{\frac{a}{4}\zeta,\frac{ar_0}{16}\kapta\big\} $ so that \eqref{small} and \ref{vr_small} hold.
By now we proved \eqref{eq:Ehatdecay} and \eqref{eq:comp}.
Finally,  notice that for small enough $\kr,v_r$ we have
\[
\kr^2 + v_\theta^4 + v_r^2 \ge \kr^4 + v_\theta^4 + v_r^4 \ge \frac{1}{3}(\kr^2 + v_\theta^2 + v_r^2)^2 \ge \frac{1}{3}\min\Big\{\frac{1}{a^2},1\Big\}\cE^2(t).
\]
We conclude, in view of  \eqref{eq:Ehatdecay} and \eqref{eq:comp},
\[
\frac{\rd}{\rd{t}}\widehat{\cE}(t) \le -\lamta \widehat{\cE}^2(t) \ \leadsto \ \widehat{\cE}(t)  \le \frac{1}{\lamta t + 1/\widehat{\cE}(0)},\quad \lamta = \frac{\lamta_1}{12}\min\left\{\frac{1}{a^2},1\right\}.
\]
 It follows that $|v_\theta|\le C\brt^{-1/2}$.

To get a better decay rate for $\kr$ and $v_r$, we use yet another modified energy functional, 
\[
\widecheck{\cE}(t):= \wU(r) + \frac{1}{2}v_r^2 - c_1 \wU'(r)v_r,
\]
 for which we find
\[
\begin{split}
\frac{\rd}{\rd{t}}\widecheck{\cE}(t) = &  -\tau \wU'(r)^2 - c_1\wU''(r)v_r^2 + \frac{v_r v_\theta^2}{r} + c_1\left(\tau \wU''(r)v_r \wU'(r) + \wU'(r)^2 - \wU'(r)\frac{v_\theta^2}{r}\right) \\
\le & -\tau \frac{a^2}{4}\kr^2 -c_1\frac{a}{2}v_r^2 + \frac{2v_r v_\theta^2}{r_0} + c_1\left(4\tau a^2  |v_r \kr| + 4a^2 \kr^2 + 4a|\kr| \frac{v_\theta^2}{r_0}\right) \\ 
\le & -\left(\tau \frac{a^2}{4} - \frac{c_1}{\kapta_1}\Big(2\tau a^2 + 4\kapta_1 a^2 + \frac{2a}{r_0}\Big) \right) \kr^2 -c_1\left(\frac{a}{2} - \kapta_1 \frac{1}{r_0} - \kapta_1\cdot 2\tau a^2\right)v_r^2 \\
& + \left( \frac{1}{c_1\kapta_1r_0} + c_1\kapta_1\frac{2a}{r_0}\right) v_\theta^4.
\end{split}
\]
By similar choices of $c_1$ and $\kapta_1$, one can guarantee that $\widecheck{\cE}(t)$ is equivalent to $\kr^2+v_r^2$, and the coefficients of $\kr^2$ and $v_r^2$ are positive. Therefore
\begin{equation}
\frac{\rd}{\rd{t}}\widecheck{\cE}(t) \le -\lamta_2 \widecheck{\cE}(t) + Cv_\theta^4 \le -\lamta_2 \widecheck{\cE}(t) + C\brt^{-2}
\end{equation}
This gives
\begin{equation}
\widecheck{\cE}(t) = e^{-\lamta_2 t}\widecheck{\cE}(0) + C\int_0^t e^{-\lamta_2(t-s)}(1+s)^{-2} \rd{s}
\end{equation}
We estimate the last integral  for large enough $t$,
\begin{equation}
\begin{split}
\int_0^t e^{-\lamta_2(t-s)}(1+s)^{-2} \rd{s} & \le  \left(\int_0^{t-\frac{1}{\lamta_2}\ln \brt} + \int_{t-\frac{1}{\mu}\ln \brt}^t\right)e^{-\lamta_2(t-s)}(1+s)^{-2} \rd{s}\\
&\le  \brt^{-1} \int_0^t (1+s)^{-2} \rd{s} +  \Big(1+(t-\frac{1}{\lamta_2}\ln \brt)\Big)^{-2}\frac{1}{\lamta_2}\ln \brt\\
& \le  \brt^{-2}  + \frac{2}{\lamta_2} \brt^{-2}\ln \brt.
\end{split}
\end{equation}
This shows that $\widecheck{\cE}(t) \le C\brt^{-2}\ln\braket{1+t}$, and therefore $ |v_r|+|\kr|\le  C\brt^{-1}\ln^{1/2}\braket{1+t}$. 
\end{proof}
Finally, we conclude by noting that the last  bound on $\kr$ tells us that
\[
\big||\bx^\tau_1(t)-\bx^\tau_2(t)|-r_0\big| \le C\brt^{-1}\ln^{1/2}\braket{1+t},\quad |\bv_1(t)-\bv_2(t)| \le C\brt^{-1/2}.
\]
Observe that this  bound on relative \emph{anticipated} positions is in fact equivalent to the claimed statement of the current positions, $\big||\bx_1(t)-\bx_2(t)|-r_0\big|\lesssim \brt^{-1}\ln^{1/2}\braket{1+t}$, which concludes the proof of theorem \ref{thm:thm3}.

\begin{remark}
Numerical examples \cite[sec. 1]{GTLW2017} show that the rate $v_\theta={\mathcal O}(t^{-\nicefrac{1}{2}})$ is optimal. Therefore, 
\[
\theta(t) = \theta_0 + \int_0^t \frac{1}{r(s)}v_\theta(s)\rd{s} = {\mathcal O}(\sqrt{t}), 
\]
which means that $\theta$ needs not converge to any point, even for small initial data. Thus, although we trace the dynamics of $\kr,v_r,v_\theta$  using  essentially perturbative arguments, the dynamics of \eqref{eq2} is not.
\end{remark}

\section{Anticipation dynamics: hydrodynamic formulation}\label{sec:hydro}
The large crowd dynamics associated with \eqref{eq:AT} is captured by the macroscopic density $\rho(t,\bx): \mathbb{R}_+\times \mathbb{R}^d \mapsto \mathbb{R}_+$ and momentum $\rho\bu(t,\bx): \mathbb{R}_+\times \mathbb{R}^d \mapsto \mathbb{R}^d$, which are governed by the hydrodynamic description \eqref{eq:hydro}
\[
\left\{\begin{split}
 \rho_t + \nabla_\bx\cdot (\rho \bu) &= 0 \\
 (\rho\bu)_t + \nabla_\bx\cdot (\rho\bu\otimes \bu) & =   - \int \nabla U(|\bx^\tau - \by^\tau|) \rd\rho(\by), \quad \bx^\tau := \bx + \tau \bu(t,\bx).
\end{split}\right.
\]
The flux on the left involves  additional second-order moment fluctuations, ${\mathcal P}$, which can be dictated by proper closure relations. As in \cite{HT2008}, we will focus on the mono-kinetic case, in which case ${\mathcal P}=0$. 

To study the large time behavior we appeal, as in the discrete case, to the basic bookkeeping of energy and enstrophy: here we consider the 
\emph{anticipated energy}
\begin{equation}
\cE(t) := \int \frac{1}{2}|\bu(t,\bx)|^2 \rho(t,\bx)\rd{\bx} + \frac{1}{2}\int\int U(|\bx^\tau(t)-\by^\tau(t)|) \rd\rho(t,\bx)\rd\rho(t,\by).
\end{equation}
Away from vacuum, the velocity field $\bu(\bx)=\bu(t,\bx)$ satisfies the transport equation
\begin{subequations}\label{eqs:trans}
\begin{equation}\label{eq:trans}
\bu_t(t,\bx) + \bu\cdot\nabla_\bx \bu(t,\bx) = \aL(\rho,\bu)(t,\bx),
\end{equation}
where  $\aL(\rho,\bu)$ denotes the anticipated interaction term 
\begin{equation}\label{eq:al}
\aL(\rho,\bu)(t,\bx):=-\int \nabla U(|\bx^\tau(t) - \by^\tau(t)|) \rd\rho(t,\by), \qquad \bx^\tau(t) = \bx + \tau \bu(t,\bx).
\end{equation}
\end{subequations}
We compute (suppressing the time dependence)
\[
\begin{split}
\frac{\rd}{\rd{t}} \cE(t) = & \int \bu(\bx)\cdot(-\bu\cdot\nabla \bu +\aL(\bx)) \rd \rho(\bx) + \int \frac{1}{2}|\bu(\bx)|^2 (-\nabla\cdot (\rho \bu))\rd{\bx} \\
& + \frac{\tau}{2}\int\int \nabla U(|\bx^\tau-\by^\tau|)\cdot ( -\bu(\bx)\cdot\nabla \bu(\bx) +\aL(\bx) +  \bu(\by)\cdot\nabla \bu(\by) -\aL(\by)) \rd\rho(\bx)\rd\rho(\by) \\
& + \frac{1}{2}\int\int U(|\bx^\tau-\by^\tau|) (-\nabla\cdot (\rho \bu)(\by)))\rd\rho(\bx)\rd{\by}\\
& + \frac{1}{2}\int\int U(|\bx^\tau-\by^\tau|) (-\nabla\cdot (\rho \bu)(\bx))\rd{\bx}\rd\rho(\by) \\
= & \int \bu(\bx)\cdot \aL(\bx) \rd\rho(\bx)+ \frac{\tau}{2}\int\int \nabla U(|\bx^\tau-\by^\tau|)\cdot (  \aL(\bx)  - \aL(\by) ) \rd\rho(\by)\rd\rho(\bx)  \\
& + \frac{1}{2}\int\int \nabla U(|\bx^\tau-\by^\tau|)\cdot (   -\bu(\by)+\bu(\bx)) \rd\rho(\by)\rd\rho(\bx)\\
= & \tau\int\int \nabla U(|\bx^\tau-\by^\tau|)\cdot  \aL(\bx)  \rd\rho(\by)\rd\rho(\bx)\\
= & -\tau \int\Big|\int  \nabla U(|\bx^\tau-\by^\tau|)  \rd\rho(\by)\Big|^2\rd\rho(\bx).
\end{split}
\]
This is the continuum analogue of the discrete enstrophy statement \eqref{energy}, which becomes apparent when it is expressed in terms of the \emph{material derivative}, 
\begin{equation}\label{eq:hydroens}
\frac{\rd}{\rd{t}} \cE(t) = -\tau \int\Big|\int  \nabla U(|\bx^\tau-\by^\tau|)  \rd\rho(\by)\Big|^2\rd\rho(\bx) = -\tau \int \left|\frac{\rD}{\rD{t}}\bu(t,\bx)\right|^2 \rd\rho(t,\bx).
\end{equation}
\subsection{Smooth solutions must flock}\label{subsec:must}
We consider the anticipation hydrodynamics \eqref{eq:hydro} with  attractive potentials, \eqref{eq:attractive}
\[
a\langle r\rangle^{-\beta} \leq \frac{U'(r)}{r}, \quad U''(r)  \leq A, \qquad 0<a<A.
\]
The study of its large time `flocking' behavior proceeds   precisely along the lines of our discrete proof of theorem \ref{thm:thm2}. Here are the three main ingredients in the proof of theorem \ref{thm:thm4}.

{\bf Step (i)} We begin where we left with the anticipated energy balance \eqref{eq:hydroens}, which we express as 
\[
\frac{\rd}{\rd{t}}\cE(t) = \int\Big|\int  \nabla U(|\bx^\tau-\by^\tau|)  \rd\rho(\by)\Big|^2\rd\rho(\bx)
 = \int\Big|\int  \frac{ U'(|\bx^\tau-\by^\tau|)}{|\bx^\tau-\by^\tau|} (\bx^\tau-\by^\tau)\rd\rho(\by)\Big|^2\rd\rho(\bx).
\]
We now appeal to the special case of lemma \ref{lem:Mean} with $\Omega=\mathbb{R}^d$ (with variable $\bx$), $\rd{\mu}=\rho(\bx)\rd{\bx}$, $\bX(\bx) = \bx^\tau, \bX(\by)=\by^\tau$ and $\ds c(\bx,\by)=\frac{U'(|\bx^\tau-\by^\tau|)}{|\bx^\tau-\by^\tau|}$, in which case we have
\begin{equation}
\int |\bx^\tau|^2 \rho(\bx)\rd{\bx} \le C(\lambda,\Lambda)\int \left| \int  \frac{U'(|\bx^\tau-\by^\tau|)}{|\bx^\tau-\by^\tau|} (\bx^\tau-\by^\tau)\rho(\by)\rd{\by} \right|^2 \rho(\bx)\rd{\bx}
\end{equation}
where $\Lambda$ and $\lambda$ are the upper- and respectively, lower-bounds of $\ds \frac{ U'(|\bx^\tau-\by^\tau|)}{|\bx^\tau-\by^\tau|}$,
\begin{equation}\label{eq:stepi}
\begin{split}
\frac{\rd}{\rd{t}}\cE(t) &= -\tau\int\Big|\int  \frac{ U'(|\bx^\tau-\by^\tau|)}{|\bx^\tau-\by^\tau|} (\bx^\tau-\by^\tau)\rd\rho(\by)\Big|^2\rd\rho(\bx)\\
& \lesssim \Big(\min \frac{ U'(|\bx^\tau-\by^\tau|)}{|\bx^\tau-\by^\tau|}\Big)^{-4}
\int\int  \big|\bx^\tau-\by^\tau\big|^2\rd\rho(\by)\rd\rho(\bx)
\end{split}
\end{equation}

{\bf Step (ii)}. A bound on the spread of the anticipated positions supported on non-vacuous states
\begin{equation}\label{eq:Spread}
\max_{\bx^\tau\in \text{supp}\,\rho(t,\cdot)}|\bx^\tau|  \leq c\brt^\eta.
\end{equation}
Arguing along the lines of lemma \ref{lem:bound} one finds that \eqref{eq:Spread} holds with $\ds \eta=\frac{1}{2(1-\beta)}$, hence 
\[
a\brt^{-\frac{\beta}{2(1-\beta)}} \lesssim \frac{U'(|\bx^\tau-\by^\tau|)}{|\bx^\tau-\by^\tau|} \le A, \qquad \bx^\tau,\by^\tau \in \text{supp}\,\rho(t,\cdot),
\]
and \eqref{eq:stepi} implies
\begin{equation}\label{eq:Enspos}
\begin{split}
\frac{\rd}{\rd{t}} \cE(t) & =  -\tau\int\Big|\int  \frac{ U'(|\bx^\tau-\by^\tau|)}{|\bx^\tau-\by^\tau|} (\bx^\tau-\by^\tau)\rd\rho(\by)\Big|^2\rd\rho(\bx) \\
& \lesssim \tau \brt^{-\frac{2\beta}{1-\beta}}\int\int  |\bx^\tau-\by^\tau|^2  \rd\rho(\bx)\rd\rho(\by).
\end{split}
\end{equation}
We are now exactly at the point we had with the discrete anticipation dynamics,
in which the decay of anticipated energy is controlled by the fluctuations of anticipated position, \eqref{eq:enspos}. 

{\bf Step (iii)}. To close the decay rate \eqref{eq:Enspos} one invokes hypocoercivity argument on the modified energy, 
\[
\widehat{\cE}(t):= \cE(t)-\epsilon(t)\int \bx^\tau\cdot\bu(\bx)\rd{\rho(\bx)}.
\]
Arguing along the lines of section \ref{sec:attraction}, one can find a suitable 
$\epsilon(t)>0$ which leads to the sub-exponential decay of  $\widehat{\cE}(t)$ and hence of the comparable $\cE(t)$, thus completing the proof of theorem  \ref{thm:thm4}.

\subsection{Existence of smooth solutions -- the 1D case}\label{subsec:1D}
We study the existence of smooth solutions of the 1D anticipated hydrodynamic system
\begin{equation}\label{eq:1Dhydro}
\left\{ \ \ \begin{split}
& \partial_t \rho + \partial_x (\rho u) = 0 \\
& \partial_t u + u\partial_x u =   - \int  U'(|x^\tau - y^\tau|) \rho(y) \rd{y}, \qquad x^\tau=x+\tau u(t,x),
\end{split}\right.
\end{equation}
subject to uniformly convex potential $U''(\cdot) \geq a>0$. 
Let $\texttt{d}:= \partial_x u$. Then
\begin{equation}\label{eq:ed}
\begin{split}
& \partial_t \texttt{d} + u\partial_x \texttt{d} + \texttt{d}^2 =  - (1+\tau \texttt{d})\int  U''(|x^\tau - y^\tau|) \rho(y) \rd{y}
\end{split}
\end{equation}
or
\begin{equation}\label{eq:quad}
\begin{split}
 \texttt{d}' = - \texttt{d}^2  - c(1+\tau \texttt{d}), \qquad {}':=\partial_t + u\partial_x,
\end{split}
\end{equation}
where by uniform convexity $c=c(t,x):=        \int  U''(|x^\tau - y^\tau|) \rho(y) \rd{y} \in [m_0 a,m_0A]$.
The discriminant of RHS,  given by
$(\tau c)^2 - 4c = c(\tau^2 c -4)$
is  non-negative, provided $\tau^2 m_0 a \ge 4$.
In this case, the smaller root of \eqref{eq:quad} is given by
\begin{equation}
\frac{1}{2}(-\tau c - \sqrt{c(\tau^2 c -4)}) \le -\frac{1}{2}(\tau m_0 a + \sqrt{m_0 a(\tau^2 m_0 a -4)}), 
\end{equation}
and the region to its right is an invariant of the dynamics \eqref{eq:ed}. We conclude the following.
\begin{proposition}[{\bf Existence of global smooth solution}]\label{prop:1Dexist} 
Consider the 1D anticipation hydrodynamic system \eqref{eq:1Dhydro} with uniformly convex potential $0< a\leq U''\leq A$. It admits a global smooth solution for sub-critical initial data, $(\rho_0,u_0)$, satisfying
\[
\min_x u'_0(x) \ge -\frac{1}{2}(\tau m_0 a + \sqrt{m_0 a(\tau^2 m_0 a -4)}), \qquad \tau \geq \frac{2}{\sqrt{m_0 a}}.
\]
\end{proposition}

\end{document}